\documentclass[a4paper, 10pt]{amsart}

\usepackage[utf8]{inputenc}
\usepackage[T1]{fontenc}
\usepackage[lighttt]{lmodern}
\usepackage{mathtools}
\usepackage{amsmath}
\usepackage{amssymb}
\usepackage{amsthm}
\usepackage[maxbibnames=99,maxalphanames=99,giveninits=true,style=alphabetic,citestyle=alphabetic,isbn=false]{biblatex}
\usepackage{hyperref}
\usepackage[shortlabels]{enumitem}
\usepackage{tikz-cd}
\usepackage[dvipsnames]{xcolor}
\usepackage{relsize} 

\tikzcdset{cramped,arrows=hook}
\DeclareSourcemap{
  \maps[datatype=bibtex]{
    \map{
      \step[fieldsource=doi,final]
      \step[fieldset=url,null]
      \step[fieldset=urldate,null]
    }  
  }
}

\hypersetup{
   colorlinks=true,
   linkcolor=Blue,
   citecolor=Blue
}

\allowdisplaybreaks[4]

\theoremstyle{plain}
\newtheorem{theorem}{Theorem}[section]

\newtheorem{lemma}[theorem]{Lemma}
\newtheorem{proposition}[theorem]{Proposition}
\newtheorem{corollary}[theorem]{Corollary}

\newtheorem{step}{Step}[theorem]

\theoremstyle{definition}
\newtheorem{definition}[theorem]{Definition}

\newtheorem{example}[theorem]{Example}
\newtheorem{question}[theorem]{Question}

\theoremstyle{remark}
\newtheorem{remark}[theorem]{Remark}

\newcommand{\conditemStyle}{\bfseries\ttfamily}
\newcommand{\algostepStyle}{\bfseries\ttfamily}

\newlist{conditions}{itemize}{1}
\newlist{algorithm}{itemize}{1}
\setlist[conditions]{font={\conditemStyle}}
\setlist[algorithm]{font ={\algostepStyle}}

\makeatletter
\newcommand{\conditem}[2]{
  \item[
  \phantomsection
  #1\protected@edef\@currentlabel{{\conditemStyle #1}}\label{#2}]
}

\newcommand{\algostep}[2]{
  \item[{
  \phantomsection
  #1\protected@edef\@currentlabel{{\algostepStyle #1}}\label{#2}}]
}
\makeatother

\DeclareMathOperator{\Aut}{Aut}

\DeclareMathOperator{\Span}{Span}
\DeclareMathOperator{\Conv}{Conv}
\DeclareMathOperator{\IntRel}{int_{rel}}
\DeclareMathOperator{\Int}{int}
\DeclareMathOperator{\CR}{CR}

\DeclareMathOperator{\End}{End}

\DeclareMathOperator{\im}{im}

\DeclareMathOperator{\diag}{diag}

\makeatletter
\newcommand*\bigcdot{\mathpalette\bigcdot@{.7}}
\newcommand*\bigcdot@[2]{\mathbin{\vcenter{\hbox{\scalebox{#2}{$\m@th#1\bullet$}}}}}
\makeatother

\newcommand{\R}{\mathbb{R}}
\newcommand{\C}{\mathbb{C}}
\newcommand{\Z}{\mathbb{Z}}
\newcommand{\N}{\mathbb{N}}

\renewcommand{\H}{\ensuremath{\mathbb{H}}}
\renewcommand{\P}{\ensuremath{\mathbb{P}}}
\newcommand{\Ell}{\mathcal{E}}
\newcommand{\A}{\mathbb{A}}
\newcommand{\D}{\mathcal{D}}
\DeclareMathOperator{\flag}{\mathcal{F}}
\DeclareMathOperator{\grass}{grass}

\newcommand{\horo}{\mathcal{H}}

\newcommand{\Cont}{\mathcal{C}}
\newcommand{\M}{\mathrm{Mat}}
\newcommand{\PGL}{\mathrm{PGL}}
\newcommand{\PO}{\mathrm{PO}}
\renewcommand{\O}{\ensuremath{\mathrm{O}}}
\newcommand{\GL}{\mathrm{GL}}
\newcommand{\SL}{\mathrm{SL}}
\newcommand{\SLpm}{\mathrm{SL}^\pm}
\newcommand{\COne}{\mathcal{C}^1}

\newcommand{\Order}{\Theta}

\DeclarePairedDelimiter\abs{\lvert}{\rvert}
\DeclarePairedDelimiter\norm{\lVert}{\rVert}

\newcommand{\tosub}[1]{\xrightarrow[#1]{}}
\newcommand{\simsub}[1]{\underset{#1}{\sim}}

\DeclarePairedDelimiter\gen{\langle}{\rangle}

\newcommand{\genRep}[2]{{\mathcal{G}_{#2}^{#1}}}
\newcommand{\periph}{\ensuremath{\mathcal{P}}}
\DeclareMathOperator{\trans}{\ell}

\makeatletter
\let\oldabs\abs
\def\abs{\@ifstar{\oldabs}{\oldabs*}}

\let\oldnorm\norm
\def\norm{\@ifstar{\oldnorm}{\oldnorm*}}
\makeatother

\newcommand*{\transpose}[2][-5mu]{\ensuremath{\mskip1mu\prescript{\smash{\mathrm t\mkern#1}}{}{\mathstrut#2}}}

\renewcommand{\epsilon}{\varepsilon}
\renewcommand{\phi}{\varphi}
\newcommand{\ie}{\emph{i.e.~}}
\newcommand{\resp}{resp.~}

\newcommand{\Def}{Def.~}

\newcommand{\Lm}{Lm.~}
\newcommand{\Prop}{Prop.~}
\newcommand{\Thm}{Thm.~}
\newcommand{\Cor}{Cor.~}

\usepackage{todonotes}

\title{On cusp holonomies in strictly convex projective geometry}
\author{Balthazar Fléchelles}
\email{balthazar.flechelles@univ-grenoble-alpes.fr}
\address{Institut Fourier, 100 rue des mathématiques, 38610, Gières, France}
\date{\today}
\thanks{The author received funding from the European Research Council (ERC) under the European Union's Horizon 2020 research and innovation programme (ERC starting grant DiGGeS, grant agreement No 715982, and ERC consolidator grant GeometricStructures, grant agreement No 614733), as well as the ANR-23-CE40-0012 HilbertXfield. The author acknowledges support of the Institut Henri Poincaré (UAR 839 CNRS-Sorbonne Université) and LabEx CARMIN (ANR-10-LABX-59-01).}

\begin{document}
\begin{abstract}
   We give a complete characterization of the holonomies of strictly convex cusps and of round cusps in convex projective geometry. We build families of generalized cusps of non-maximal rank associated to each strictly convex or round cusp. We also extend Ballas--Cooper--Leitner's definition of generalized cusp to allow for virtually solvable fundamental group, and we produce the first such example with non-virtually nilpotent fundamental group.

   Along with a companion paper, this allows to build strictly convex cusps and generalized cusps whose fundamental group is any finitely generated virtually nilpotent group. This also has interesting consequences for the theory of relatively Anosov representations.

\end{abstract}
\maketitle

\tableofcontents

\section{Introduction}

A \emph{cusp} of a hyperbolic orbifold $\H^n/\Gamma$ is a suborbifold that is the quotient of a horoball of $\H^n$ by a \emph{maximal parabolic subgroup} of $\Gamma$, \ie a maximal infinite subgroup of $\Gamma$ that fixes exactly one point in $\partial\H^n$. They appear in particular in finite volume and geometrically finite hyperbolic orbifolds, and play a central role in their study. In this article, we are interested in understanding cusps in a generalization of hyperbolic geometry, namely \emph{strictly convex (projective) geometry}.

The Kazhdan-Margulis lemma implies that hyperbolic cusp holonomies are virtually nilpotent. In fact, the Bieberbach theorem, combined with the fact that hyperbolic horospheres are endowed with a Euclidean affine structure induced by the ambient metric, show that hyperbolic cusp holonomies are virtually abelian. This leads to a precise classification of hyperbolic cusps and their holonomies, which in turns allows for a precise understanding of geometrically finite hyperbolic orbifolds. In order to define geometric finiteness, we need to recall the \emph{thick-thin decomposition} theorem.

A $n$-dimensional hyperbolic orbifold can be decomposed into its \emph{thick part}, where the injectivity radius is at least the Margulis constant $\mu_n$, and its complement, called the \emph{thin part}. The \emph{thick-thin decomposition theorem} states that the thin part decomposes into a union of connected components with disjoint closures and simple geometry. The bounded components
, called \emph{Margulis tubes}, are tubular neighborhoods of simple closed geodesics of length smaller than $\mu_n$. The unbounded components, which are the intersections of the orbifold's cusps with the thin part, are roughly quotients of uniform neighborhoods of $m$-dimensional horoballs of $\H^n$ by the cusp holonomy, for some $m\leq n$. A hyperbolic orbifold is \emph{geometrically finite} if its thick part interior is homeomorphic to the interior of a compact manifold with boundary. In particular, geometrically finite hyperbolic orbifolds have finitely many cusps and Margulis tubes. A geometrically finite hyperbolic orbifold that has no cusps is said to be \emph{convex-cocompact}. 

Geometrically finite and convex-cocompact hyperbolic orbifolds behave similarly to respectively finite volume and compact hyperbolic orbifolds in many ways, and have been extensively studied. A key difference between the geometrically finite case and the finite covolume case is that all cusps of a finite volume hyperbolic orbifold have \emph{maximal rank}, meaning that the cusp holonomy acts cocompactly on the horospheres it preserves. This does not have to be the case for geometrically finite hyperbolic orbifolds.

\subsection{Strictly convex geometry}

At the end of the nineteenth century, Hilbert observed that on any \emph{properly convex domain} $\Omega\subset\P(\R^n)$, \ie a bounded convex open subset of an affine chart of $\P(\R^n)$, one can define a complete and proper metric $d_\Omega$ by means of the cross-ratio. This metric, called the \emph{Hilbert metric}, induces a Finsler metric on the tangent bundle, and has the property that all projective segments are geodesics. Moreover, since it is defined using the cross-ratio, the Hilbert metric is invariant under all projective transformations $g\in\PGL(n,\R)$ which preserve $\Omega$. Such isometries of $(\Omega,d_\Omega)$ are called \emph{automorphisms of $\Omega$}.

The regularity of the boundary $\partial\Omega$ of $\Omega$ profoundly impacts the properties of the Hilbert metric. When $\partial\Omega$ contains non-trivial segments, the geometry recalls that of non-positively curved Riemannian spaces. If for instance $\Omega$ is a $n-1$-dimensional simplex in $\P(\R^n)$, then $(\Omega,d_\Omega)$ is isometric to $\R^{n-1}$ endowed with a polytopal norm.

On the other hand, when $\Omega$ is \emph{strictly convex}, meaning that $\partial\Omega$ contains no non-trivial segments, the geometry resembles more that of hyperbolic space. In fact, if $\Omega$ is a ball, then $(\Omega,d_\Omega)$ is isometric to $\H^{n-1}$.

A strictly convex orbifold is the quotient of a strictly convex domain $\Omega$ by a discrete subgroup of automorphisms. Round orbifolds, a particular subclass of strictly convex orbifolds will also play a role in this work. They are quotients of a \emph{round domain}, \ie a strictly convex domain whose boundary has $\COne$ regularity, by a discrete subgroup of automorphisms. Observe that hyperbolic geometry is not only an example of strictly convex geometry, but also of round geometry. 

Things are very close to the hyperbolic picture in strictly convex geometry. A Kazhdan-Margulis lemma holds, as well as a thick-thin decomposition theorem (see \cite{cramponMarquis2013,cramponMarquis2014,cooperLongTillmann2015}). This allows to define geometric finiteness in strictly convex geometry as in the hyperbolic case above. Geometrically finite round orbifolds were first defined and extensively studied in \cite{cramponMarquis2014,blayacMarquis2024}, and the authors showed that numerous characterizations from the hyperbolic case generalized to this setting.

Just like in the hyperbolic setting, a cusp of a strictly convex orbifold $\Omega/\Gamma$ is a suborbifold that is the quotient of a horoball of $\Omega$ by a maximal parabolic subgroup of $\Gamma$, \ie a maximal infinite subgroup of $\Gamma$ that fixes exactly one point in $\partial\Omega$.
Apart from the maximal rank strictly convex cusps, whose holonomy were proven to be conjugated to maximal rank hyperbolic cusp holonomies of the same dimension in \cite{cooperLongTillmann2015} (see also \cite{cramponMarquis2014} for the round case), the understanding of strictly convex cusp holonomies seems to be a blind spot of the literature. This is surprising as there are a few indications showing the importance of this question.

Firstly, there appears to be a deep connection between strictly convex geometry and generalizations of rank 1 phenomena to the higher rank setting. The most striking example is the proof in \cite{dancigerGueritaudKassel2024} that Anosov representations, a generalization of rank 1 convex-cocompactness, can be seen as convex-cocompact actions on strictly convex domains (a weaker but related statement is proven in \cite{zimmer2021}). One may think that adding cusps into the picture would allow to get similar results for relatively Anosov representations, a generalization of Anosov representations to relatively hyperbolic groups. This is the route taken by Islam, Zhu and the author in \cite{mitulFeng1} to prove that relatively Anosov representations correspond to geometrically finite actions on round domains, using the characterizations of strictly convex cusp holonomies developed here (see Theorem \ref{thmIntro:characHoloStrictlyConvCusps}).

Moreover, a better understanding of strictly convex cusp holonomies may allow to generalize theorems about geometrically finite hyperbolic orbifolds to the strictly convex realm. In fact, the aforementioned classification of maximal rank strictly convex cusps in \cite{cramponMarquis2014,cooperLongTillmann2015} allowed Cooper--Long--Tillmann to extend numerous classical results on finite volume hyperbolic orbifolds to the strictly convex case.

However, there is an important difference between the hyperbolic and the strictly convex settings: because strictly convex horospheres typically are not endowed with a natural Euclidean affine structure, one cannot use the Bieberbach theorem to conclude that cusp holonomies must be virtually abelian. It is then natural to ask whether the Margulis lemma is the only restriction on the fundamental groups of strictly convex cusps.

\begin{question}\label{Q:vNilpGpsAreFundGpsOfRoundCusps}
   Can any virtually nilpotent group be realized as the fundamental group of a strictly convex cusp? Of a round cusp?
\end{question}

Of course, our original question about deepening the understanding of strictly convex cusp holonomies remains.

\begin{question}\label{Q:classificationOfCuspHolonomies}
   Can we classify or characterize the holonomies of strictly convex cusps? What about the round case?
\end{question}

In a series of two articles, we set out to answer these two questions. The present article shall tackle Question \ref{Q:classificationOfCuspHolonomies}, while the companion paper \cite{flechelles2025unipotentReps} brings a positive answer to Question \ref{Q:vNilpGpsAreFundGpsOfRoundCusps}.

\subsection{Characterization of strictly convex cusp holonomies}

Before stating our results, we introduce divergent groups and their limit sets. We will denote by $\flag_{1,n-1}(\R^n)$ the space of line-hyperplane flags in $\R^n$, and by $\pi_1:\flag_{1,n-1}(\R^n)\to\P(\R^n)$ the natural projection.
\begin{definition}\label{def:divergentGroup}
   A discrete group $\Gamma<\PGL(n,\R)$ is $P_1$-divergent if for any injective sequence $(\gamma_n)_n$ in $\Gamma$, there is, up to extracting a subsequence, a pair of flags $(p_\pm,H_\pm)\in\flag_{1,n-1}(\R^n)$ such that for any flag $(p,H)$ transverse to $(p_-,H_-)$, 
   
   $\gamma_n (p,H)\to(p_+,H_+)$ uniformly on compact sets.
\end{definition}
The flags $(p_+,H_+)$ and $(p_-,H_-)$ are respectively called the \emph{attractive} and \emph{repulsive} flags of the sequence $(\gamma_n)_n$. Similarly, $p_+$ and $H_-$ are respectively called the \emph{attractive point} and \emph{repulsive hyperplane} of $(\gamma_n)_n$ in $\P(\R^n)$.

We can now define the limit sets of $\Gamma$ in $\flag_{1,n-1}(\R^n)$ and $\P(\R^n)$:
\begin{align*}
   \Lambda_{1,n-1}(\Gamma) &:= \{\text{attractive flags of sequences in $\Gamma$}\}\subset\flag_{1,n-1}(\R^n) \\
   \Lambda_1(\Gamma) &:= \{\text{attractive points of sequences in $\Gamma$}\}\subset\P(\R^n)
\end{align*}
Observe that $\pi_1(\Lambda_{1,n-1}(\Gamma))=\Lambda_1(\Gamma)$. They are non-empty closed $\Gamma$-invariant subsets of $\flag_{1,n-1}(\R^n)$ and $\P(\R^n)$ respectively.

It is a classical fact that a group preserving a strictly convex domain is $P_1$-divergent (see Proposition \ref{prop:autOfStrictlyConvDomainsAreP1Div}). Hence discrete groups of isometries of $\H^n$ are always $P_1$-divergent. It turns out hyperbolic cusp holonomies are exactly the discrete groups $\Gamma$ of isometries of $\H^n$ whose limit sets $\Lambda_1(\Gamma)$ and $\Lambda_{1,n-1}(\Gamma)$ are reduced to a point. This motivates the following characterization of strictly convex and round cusp holonomies.

\begin{theorem}[see \Thm \ref{thm:characHoloStrictlyConvCusps} \& \ref{thm:characHoloRoundCusps}]\label{thmIntro:characHoloStrictlyConvCusps}
   Let $\Gamma$ be a discrete subgroup of $\PGL(n,\R)$. Then $\Gamma$ is the holonomy of a strictly convex (\resp round) cusp if and only if $\Gamma$ preserves a properly convex domain, is $P_1$-divergent, and its limit set $\Lambda_1(\Gamma)$ (\resp $\Lambda_{1,n-1}(\Gamma)$) is reduced to a point. 
\end{theorem}

Note that Theorem \ref{thmIntro:characHoloStrictlyConvCusps} does not appear to give more constraints on the fundamental groups of strictly convex and round cusps,
so Question \ref{Q:vNilpGpsAreFundGpsOfRoundCusps} remains open. It is also unclear at first sight wether or not strictly convex cusp holonomies are identical to round cusp holonomies. In the companion paper \cite{flechelles2025unipotentReps}, we show that there are strictly convex cusp holonomies that are not round cusp holonomies.

\begin{theorem}[see \cite{flechelles2025unipotentReps}]\label{thmIntro:nilpGpsElemCuspExamples}
   Let $\Gamma$ be a finitely generated virtually nilpotent group. Then $\Gamma$ is the fundamental group of a round cusp.

   It is also the fundamental group of a strictly convex cusp whose holonomy is not the holonomy of a round cusp.
\end{theorem}

The only previously known example of a non-virtually abelian group being the fundamental group of a strictly convex cusp is the Heisenberg group (see \cite{cooper2017}). We generalize Cooper's construction in order to build representations of virtually nilpotent groups satisfying the conditions of Theorem \ref{thmIntro:characHoloStrictlyConvCusps}.

The cusps we build in Theorems \ref{thmIntro:characHoloStrictlyConvCusps} and \ref{thmIntro:nilpGpsElemCuspExamples} are all \emph{elementary cusps}, meaning that the larger manifold of which they are the cusp has the same holonomy as the cusp itself.

It is not clear at first that non-elementary cusp holonomies are the same as elementary cusp holonomies in strictly convex geometry.

In hyperbolic geometry, it is not difficult to show that there is no real difference between elementary and non-elementary cusps, by playing ping-pong with a hyperbolic isometry.
We cannot naively apply this strategy in strictly convex geometry because there is no guarantee that the strictly convex domains invariant under our cusp holonomies admit any hyperbolic automorphisms. Instead, we use \emph{relatively Anosov representations} and tools developed by Islam, Zhu and the author in \cite{mitulFeng1}, including the generalization of a ping-pong argument from \cite{zhuZimmer2024Examples}.

\subsection{Non-elementary cusps and relatively Anosov representations}

Relatively Anosov representations are a generalization of Anosov representations which allows for relatively hyperbolic groups at the source. It generalizes rank $1$ geometric finiteness to the higher rank setting similarly to how Anosov representations generalize rank $1$ convex-cocompactness to the higher rank. Our relatively Anosov representations are the same as the relatively Anosov representations of \cite{zhuZimmer2024Examples,zhuZimmer2025Theory}, the asymptotically embedded representations of \cite{kapovichLeeb2023}, and the relatively dominated representations of \cite{zhu2021,zhu2023}. We will also need in this paper a new more general version of relatively Anosov representations, called \emph{asymmetrically relatively Anosov representations}, which are introduced and studied in \cite{mitulFeng1}.

\begin{definition}
   Let $\Gamma<\PGL(n,\R)$ be a $P_1$-divergent discrete subgroup. We say that $\Gamma$ is \emph{transverse} if, for all $x\neq y$ in $\Lambda_{1,n-1}(\Gamma)$, $x$ and $y$ are transverse flags.

   We instead say that $\Gamma$ is \emph{$\pi_1$-transverse} if, for all $x,y\in\Lambda_{1,n-1}(\Gamma)$ such that $\pi_1(x)\neq\pi_1(y)$, $x$ and $y$ are transverse.
\end{definition}

The definition of (asymmetrically) relatively Anosov representations involves the notion of geometrically finite convergence actions on compact metrizable spaces. Such actions generalize the action of a geometrically finite hyperbolic group on its limit set in $\partial\H^n$
. For the definition and more context, see for instance \cite{bowditch1999,bowditch2012}.
\begin{definition}
   A discrete finitely generated group $\Gamma<\PGL(n,\R)$ is \emph{$P_1$-Anosov relative to} a collection $\periph$ of subgroups of $\Gamma$ if it is $P_1$-divergent, transverse, and the induced convergence action of $\Gamma$ on $\Lambda_{1,n-1}(\Gamma)$ is geometrically finite with parabolic subgroups $\periph$.
\end{definition}
\begin{definition}
   A discrete finitely generated group $\Gamma<\PGL(n,\R)$ is \emph{asymmetrically $P_1$-Anosov relative to} a collection $\periph$ of subgroups of $\Gamma$ if it is $P_1$-divergent, $\pi_1$-transverse, and the induced convergence action of $\Gamma$ on $\Lambda_1(\Gamma)$ is geometrically finite with parabolic subgroups $\periph$.
\end{definition}

The groups in the collection $\periph$ with respect to which a group $\Gamma$ is (asymmetrically) relatively $P_1$-Anosov are called \emph{peripheral subgroups} of $\Gamma$. An important observation is that by Theorem \ref{thmIntro:characHoloStrictlyConvCusps}, round (\resp strictly convex) cusp holonomies coincide with the peripheral subgroups of relatively $P_1$-Anosov representations (\resp asymmetrically relatively $P_1$-Anosov representations) which preserve a properly convex domain. This allows us to make use of results in \cite{mitulFeng1}, yielding:
\begin{theorem}
   Let $\Gamma<\SLpm(n,\R)$ be the holonomy of a strictly convex (resp. round) elementary cusp. Then $\Gamma$ is the holonomy of a non-elementary strictly convex (\resp round) cusp.
   
\end{theorem}

Moreover, Theorem \ref{thmIntro:nilpGpsElemCuspExamples} and the results of \cite{mitulFeng1} also allow to build a lot of new relative and asymmetrically relatively $P_1$-Anosov representations, further showing the strength of the geometric method in the study of Anosov representations.
\begin{theorem}\label{thmIntro:existenceOfRelAnRepsWithPeriphSgpsAnyVNilpGp}
   Let $\Gamma$ be a finitely generated virtually nilpotent group. Then there exist (asymmetrically) relatively $P_1$-Anosov representations with a peripheral subgroup isomorphic to $\Gamma$.
\end{theorem}

\subsection{The non-strictly convex setting: generalized cusps}

As observed in \cite{cooperLongTillmann2015}, no naive generalization of the thick-thin decomposition theorem can hold in general convex projective geometry, as there are examples for any $\epsilon>0$ of convex projective manifolds whose injectivity radius at all points is smaller than $\epsilon$. This implies that if $\epsilon$ is small enough, the thin part equals the full manifold, and therefore fails to decompose as a union of connected components whose holonomy is virtually nilpotent. This implies that there is no good definition of convex projective cusps at the moment, if it is even possible. Nevertheless, \emph{generalized cusps} were introduced and studied in \cite{cooperLongTillmann2018,ballasCooperLeitner2020,ballasCooperLeitner2022} by analogy with a classical characterization of cusps in pinched negatively curved Riemannian geometry. 

\begin{definition}\label{defIntro:genCusps}
   A generalized cusp is a closed suborbifold $X$ of a convex projective orbifold $\Omega/\Gamma$ such that
   \begin{enumerate}
      \item $\Gamma$ is virtually nilpotent;
      \item $\Omega$ is foliated by a family of strictly convex $\Gamma$-invariant hypersurfaces $(S_t)_{t\in\R}$ such that $S_t\subset\Conv S_s$ whenever $s<t$;
      \item $\Omega/\Gamma$ is diffeomorphic to $S_0/\Gamma\times\R_{\geq 0}$;
      \item $X$ is the projection in $\Omega/\Gamma$ of $\cup_{t\geq 0}S_t=(\Conv S_0)/\Gamma$.
   \end{enumerate}
\end{definition}

Ballas--Cooper--Leitner gave in \cite{ballasCooperLeitner2020} a classification of all generalized cusps of maximal rank, meaning that $S_0/\Gamma$ is compact. They were able to show that in this case, $\Gamma$ must be virtually abelian, much like maximal rank strictly convex cusps have virtually abelian holonomies. Though maximal rank hyperbolic cusps are examples of maximal rank generalized cusps, general cusps are typically non strictly convex. Indeed, there is always a (closed) simplex $\Delta\subset\partial\Omega$ of dimension $d\leq \dim\Omega- 1$, which we call the \emph{boundary simplex} of the generalized cusp, such that $\partial\Omega-\Delta$ is strictly convex. Hence a maximal rank generalized cusp is strictly convex if and only if $d=0$ and $\Delta$ is a point, in which case the maximal rank assumption implies that $\Gamma$ is the holonomy of a hyperbolic cusp.

We generalize this geometric description of the maximal rank generalized cusp and show that it always gives rise to generalized cusps (see paragraph \ref{par:genCuspGeom}). We use this strategy in order to produce examples of non-maximal rank generalized cusps.

\begin{theorem}[see \Prop \ref{prop:geomGenCuspsAreGenCusps} \& \Cor \ref{cor:genCuspDomainsForVirtNilpGps}]\label{thmIntro:strCvxCuspHolYieldGenCusps}
   Let $\Gamma<\PGL(n,\R)$ be the holonomy of a strictly convex cusp. Then for all $s\in\N$ and any subgroup $\Gamma'<\Gamma\times\Z^s$, there is a generalized cusp of dimension $n+s$ with a boundary simplex of dimension $s$ whose fundamental group is $\Gamma'$.
\end{theorem}

Combined with Theorem \ref{thmIntro:nilpGpsElemCuspExamples}, this gives the first examples of generalized cusps whose holonomy is not virtually abelian.

\begin{corollary}
   There are (non strictly convex) generalized cusps whose holonomy is isomorphic to any finitely generated virtually nilpotent group.
\end{corollary}

Moreover, we show that relaxing slightly Definition \ref{defIntro:genCusps} by allowing $\Gamma$ to be virtually solvable rather than virtually nilpotent gives new examples. Observe that by the Tits alternative, it is natural to ask for virtual solvability rather than virtual nilpotency as this is sufficient for making sure the holonomy is (virtually) elementary (see Proposition \ref{prop:solvGpsAreElem}).

\begin{theorem}[see \Prop \ref{prop:geomGenCuspsAreGenCusps} \& \ref{prop:exOfGenCuspWithSolvableHolo}]
   Let $G$ be the group of upper triangular matrices in $\SL(2,\Z[\sqrt{2}])$. Then $G$ is solvable non-virtually nilpotent, and there exists a generalized cusp of dimension $6$ with boundary simplex of dimension $2$ whose fundamental group is $G$.
\end{theorem}

This is the first example of a generalized cusp whose fundamental group is not virtually nilpotent. We expect that any virtually polycylic group is the fundamental group of a generalized cusp.
Such examples show that if there is a well-defined concept of cusp in convex projective geometry, and if generalized cusps are indeed cusps, then even the Kazhdan-Margulis lemma does not always apply to cusp holonomies. Hence a theory of convex projective cusps if it exists ought to involve arguments that drastically depart from the classical case of hyperbolic geometry.

\subsection{Organization of the paper}

In Section 2, we give reminders about convex projective geometry. We then recall results from \cite{cooperLongTillmann2015} about elementary groups, and prove a smoothing lemma for elementary groups in Section 3. Section 4 introduces notions used to study the asymptotics and properness of sets of functions, and Section 5 recalls basic results about $P_1$-divergent groups. In Section 6, we state and prove Theorem \ref{thmIntro:characHoloStrictlyConvCusps}, and Section 7 is devoted to the application of this result to the building of new examples of generalized cusps.

\subsection{Acknowledgments}

We are grateful to Sami Douba for interesting us in Cooper's article \cite{cooper2017} and for suggesting possible applications of the results in this article. We would like to warmly thank Fanny Kassel, Gye-Seon Lee, Daryl Cooper and Olivier Guichard for their precious help and feedback while reviewing the earliest version of this work. Our thanks also go to Yves Benoist, Anne Parreau and Pierre-Louis Blayac for helpful discussions and interesting questions.

\section{Reminders about convex projective geometry}

In this section, we recall classical results in convex projective geometry that are relevant to this work. 
A very good survey of most of the results in this section is given in \cite{marquis2014}. An excellent introduction that focuses more on the metric aspects of convex projective geometry is \cite{vernicos2005} (in french).

\subsection{The Hilbert metric}

The starting point of convex projective geometry are \emph{properly convex domains}.

\begin{definition}
   Let $C$ be a convex cone of $\R^n$. We say that $C$ is \emph{sharp} if its closure does not contain a full line.

   A domain $\Omega$ in $\P(\R^n)$ is said to be \emph{properly convex} if it is the projection of a sharp (open) cone of $\R^n$.
\end{definition}

Alternatively, one can think of properly convex domains as open subsets of $\P(\R^n)$ that are convex and bounded in some affine chart. This allows to easily define the \emph{open faces of $\Omega$}.

\begin{definition}
   Given $x,y\in\overline{\Omega}$, we say that $x$ and $y$ \emph{are in the same open face} if $x=y$ or there is a segment in $\overline{\Omega}$ containing both $x$ and $y$ in its interior. This defines an equivalence relation on $\overline{\Omega}$. The \emph{open faces of $\Omega$} are the equivalence classes of this relation.

   The largest open face of $\Omega$ is $\Omega$ itself, and the other faces form a partition of its boundary. We may say open faces of $\partial\Omega$ for the open faces of $\Omega$ that are not $\Omega$.
\end{definition}

In the following, all mentions of faces of $\Omega$ denote open faces, unless stated otherwise. The points of $\partial\Omega$ that are open faces of $\Omega$ will play a special role in our analysis; they are said to be \emph{extreme}.

\begin{definition}
   A point $x$ of a closed convex set $X$ is \emph{extreme} if it does not lie in the interior of any segment of $X$. It is \emph{strongly extreme} if it does not lie in any non-trivial segment of $X$.
\end{definition}

The convex hull of a subset of an affine space is the smallest convex subset that contains it. We will need to consider convex hulls in $\P(\R^n)$. These are a priori not well defined, as there are always two different projective segments between two distinct points in $\P(\R^n)$. Observe however that if $\Omega\subset\P(\R^n)$ is a properly convex domain, there is a unique segment $[x,y]\subset\overline{\Omega}$ between any pair of distinct points $x,y\in\overline{\Omega}$. Moreover, in any affine chart containing $\Omega$, the choice of segments given by the affine chart and by $\overline{\Omega}$ coincides. Therefore, the convex hull of subsets $X\subset\overline{\Omega}$ is well-defined independently of the choice of affine chart containing $\overline{\Omega}$.

We will mostly find ourselves in this situation, so we will simply write $\Conv$ to denote the convex hull in any affine chart containing the properly convex domain. When we are not in this context, we will make it explicit by mentioning that we take the convex hull in some choice of affine chart.

We now give the definition of the \emph{Hilbert metric} on a properly convex domain $\Omega$. For $x,y\in\Omega$ two distinct points, let
\begin{equation*}
   d_\Omega(x,y) := \frac12\log\CR(a,x,y,b)
\end{equation*}
where $a$ and $b$ are the two intersection points of the projective line $L$ going through $x$ and $y$ with the boundary of $\Omega$, arranged so as to be in the order $a,x,y,b$ along $L$. The quantity $\CR(a,x,y,b)$ is the cross-ratio of these four points, and equals
\begin{equation*}
   \frac{\norm{a-y}}{\norm{a-x}}\cdot\frac{\norm{b-x}}{\norm{b-y}}
\end{equation*}
in any affine chart of $L$ containing the four points, and for any norm on this affine chart.

The topology on $\Omega$ defined by the Hilbert metric is the trace topology of the inclusion of $\Omega$ in $\P(\R^n)$. The pair $(\Omega,d_\Omega)$ is called a \emph{Hilbert geometry}. It is very rarely an example of Riemannian geometry, but it is always a Finsler geometry \cite[\Prop 1 \& \Thm 2]{vernicos2005}. An important fact is that projective segments are geodesics for the Hilbert metric. However, geodesics are not always given by line segments (see for instance \cite[\Prop 3 \& 4]{vernicos2005}).

It is not hard to see that cross ratios are invariant under projective transformations. It follows that the elements of
\begin{equation*}
   \Aut(\Omega) := \{g\in\PGL(n,\R),g\Omega=\Omega\}
\end{equation*}
are isometries of the Hilbert metric on $\Omega$. They are called \emph{automorphisms of $\Omega$}.

It is apparent that the regularity of the Hilbert metric on $\Omega$ largely depends on the regularity of the boundary of $\Omega$. For this reason, regularity hypotheses on Hilbert geometries are usually stated as regularity hypotheses on the boundary of the domains.

\begin{definition}
   A properly convex domain $\Omega$ is \emph{strictly convex} if there is no non-trivial segment in $\partial\Omega$. It is \emph{$\COne$ at a point $x\in\partial\Omega$} if there is a unique hyperplane supporting $\Omega$ at $x$. We say that \emph{$\Omega$ is $\COne$} if it is $\COne$ at all points of its boundary.

   Moreover, we say that $\Omega$ is a \emph{round domain} if it is both strictly convex and $\COne$.
\end{definition}

In the following, we will need a less standard variation on the notion of $\COne$ points, which previously appeared for simplices in the thesis of Adva Wolf \cite{wolf2020}.

\begin{definition}
   A face $\omega$ of $\Omega$ is $\COne$ if there is a unique supporting hyperplane of $\Omega$ that contains $\omega$.
\end{definition}

It is often practical to consider a sharp open convex cone $C$ lifting a given properly convex domain $\Omega$.

Upon doing so, one can also lift the automorphisms of $\Omega$ to $\SLpm(n,\R)$. There is a priori a choice to be made in the lift, but observe that once $C$ is fixed, there is a unique lift of any given automorphism of $\Omega$ to $\SLpm(n,\R)$ which preserves $C$. In particular, the automorphism group of $\Omega$ lifts uniquely to a subgroup of $\SLpm(n,\R)$ which preserves $C$. For this reason, we will not make a distinction in what follows between a subgroup $\Gamma$ of $\PGL(n,\R)$ preserving $\Omega$ and its unique lift to $\SLpm(n,\R)$ preserving $C$, and we will denote both groups by $\Gamma$.

Benzécri's theorem is one of the most fundamental observations in convex projective geometry, and has a lot of interesting consequences. Before giving its statement, consider the space $C(\R^n)$ of the properly convex sets in $\P(\R^n)$, endowed with the Hausdorff metric, and $C_\ast(\R^n)$ the set of pointed properly convex domains endowed with the product topology. 

\begin{theorem}[\cite{benzecri1960}]\label{thm:Benzecri}
   The natural action of $\PGL(n,\R)$ on $C_\ast(\R^n)$ is proper and cocompact.
\end{theorem}

One of the consequences of Benzécri's theorem is the following fact, which explains why convex projective geometry is an excellent tool for studying the discrete subgroups of $\PGL(n,\R)$.

\begin{proposition}\label{prop:autGpActsProp}
   Let $\Omega$ be a properly convex domain of $\P(\R^n)$. Then the automorphism group $\Aut(\Omega)$ acts properly on $\Omega$.

   In particular, a group $\Gamma$ of automorphisms of $\Omega$ is discrete if and only if it acts properly discontinuously on $\Omega$.
\end{proposition}
\begin{proof}
   First, observe that the Hilbert metric is proper, meaning that all closed balls are compact subsets of $\Omega$. Then, seeing the action of $\Aut(\Omega)$ on $\Omega$ as the restriction of the action of $\PGL(n,\R)$ on $C_\ast(\R^n)$, Benzécri's theorem \ref{thm:Benzecri} implies that the action is proper.

   In particular, for any $x\in\Omega$, and any $R>0$, the set of automorphisms $g\in\Aut(\Omega)$ such that $gx\in \overline{B(x,R)}$ is a compact subset of $\Aut(\Omega)$.
\end{proof}

It follows that the \emph{full orbital limit set}
\begin{equation*}
   \Lambda_\Omega(\Gamma) := \bigcup_{o\in\Omega}\overline{\Gamma\cdot o}-\Gamma\cdot o
\end{equation*}
of a discrete subgroup $\Gamma$ of $\Aut(\Omega)$ must lie in $\partial\Omega$.

Convex projective geometry is the set of all geometries (in the sense of the Erlangen program) modeled on some Hilbert geometry. It is a classical observation that the projective model of $\H^n$ makes it a Hilbert geometry. A less straightforward observation is that the Hilbert geometry on the $n$-dimensional simplex is isometric to $\R^n$ endowed with a polytopal norm \cite{delaharpe1993}. It is also possible to realize the symmetric space of $\SL(n,\R)$ as a properly convex domain, the Hilbert metric being equivalent to the Riemannian metric. Therefore, convex projective geometry is a very versatile setup, that displays behaviors from the worlds of both negatively curved and flat manifolds.

In fact, any discrete linear group can be realized as the holonomy of a convex projective orbifold, by having it act on the projective model of the symmetric space of $\SL(n,\R)$, and using Proposition \ref{prop:autGpActsProp}.

\subsection{Duality}
One of the most important tools in convex projective geometry is duality. If $C\subset\R^n$ is an open convex cone, we can define its dual, which is an open convex cone of $\P((\R^n)^\ast)$
\begin{equation*}
   C^\ast = \left\{\alpha\in(\R^n)^\ast, \alpha(x) < 0,\forall x\in \overline{C}-\{0\}\right\}
\end{equation*}

The following lemma will be very useful later on. It implies that properly convex domains are exactly the open subsets of projective space that are bounded and convex in some affine chart.

\begin{lemma}\label{lm:sharpConvConeNSC}
   Let $C$ be an open cone in $\R^n$. Then $C$ is sharp if and only if $C^\ast$ is non-empty.
\end{lemma}
\begin{proof}
   The open cone $C$ is sharp if and only if its closure $\overline{C}$ does not contain any line. It is clear that if $\overline{C}$ contains a line, then there is no linear form $\alpha$ taking strictly negative values on $\overline{C}$. If $\overline{C}$ does not contain any line, we must have $\overline{C}\cap -\overline{C}=\{0\}$. Since a closed convex set of $\R^n$ is the intersection of the affine half-spaces that contain it, we have
   \begin{equation*}
      \overline{C}=\bigcap_{\substack{\alpha\in(\R^n)^\ast \\\alpha(\overline{C})\leq 0}}\{x\in\R^n,\alpha(x)\leq 0\}
   \end{equation*}
   so that
   \begin{align*}
      \overline{C}\cap-\overline{C} &= \bigcap_{\substack{\alpha\in(\R^n)^\ast\\\alpha(\overline{C})\leq 0}}\{x\in\R^n,\alpha(x)\leq 0\}\cap\{x\in\R^n,\alpha(x)\geq 0\}\\
      &= \bigcap_{\substack{\alpha\in(\R^n)^\ast\\\alpha(\overline{C})\leq 0}}\ker\alpha
   \end{align*}

   It follows that we can find $\alpha_1,\dots,\alpha_k\in(\R^n)^\ast$ taking non-positive values on $\overline{C}$ such that $\ker\alpha_1\cap\dots\cap\ker\alpha_k=\{0\}$. Now, if we let $\alpha=\alpha_1+\dots+\alpha_k$, we have $\alpha(x)\leq 0$ for all $x\in\overline{C}$, with equality if and only if $\alpha_1(x)=\dots=\alpha_k(x)=0$, if and only if $x=0$. Hence $\alpha\in C^\ast$, so the dual is non-empty.
\end{proof}

In restriction to sharp open convex cones, duality is a well-defined involution.

\begin{lemma}\label{lm:dualIsInvolution}
   Let $C$ be a sharp open convex cone. Then $C^{\ast\ast} = C$.
\end{lemma}
\begin{proof}
   If $x\in C$, then for any $\alpha\in C^\ast$, $\alpha(x) < 0$, so that $\alpha(x)\leq 0$ for all $\alpha\in \overline{C^\ast}$. Since the last inequality holds for any $x\in C$ which is open and convex, we must have $\alpha(x) < 0$ for all $\alpha\in\overline{C^\ast}-\{0\}$, that is $x\in C^{\ast\ast}$.

   For the converse, observe that since $C$ is sharp, $\P(C)$ is bounded and convex in some affine chart $\A$ by Lemma \ref{lm:sharpConvConeNSC}. Hence, for any $x\not\in \overline{C}$, we can find $\alpha\in C^\ast$ such that $\alpha(x)>0$ by choosing an affine hyperplane separating $[x]$ from $\overline{\P(C)}$ in $\A$. It follows that $C^{\ast\ast}\subset C$ as $\alpha(x) < 0$ for all $\alpha\in\overline{C^\ast}-\{0\}$ and $x\in C^{\ast\ast}$. Since $C$ and $C^{\ast\ast}$ are both open convex cones, we conclude that $C\subset C^{\ast\ast}$.
\end{proof}

We extend duality to properly convex domains of $\P(\R^n)$ by lifting them to sharp open cones of $\R^n$. That is, if $\Omega$ is the projection of a sharp open cone $C$, then $\Omega^\ast$ is the projection of $C^\ast$. Observe that $\Aut(\Omega)$ naturally acts on $\Omega^\ast$ by transposition.

Since there is a bijection between projective hyperplanes of $\P(\R^n)$ and points in $\P((\R^n)^\ast)$, it is often useful to think of $\Omega^\ast$ as the set of projective hyperplanes that do not meet the closure of $\Omega$. Equivalently, $\Omega^\ast$ is the set of hyperplanes defining an affine chart in which $\Omega$ is bounded.

The boundary points of $\Omega^\ast$ can also be interpreted as supporting hyperplanes to $\Omega$. This leads to the following observation.

\begin{lemma}
   Let $\Omega$ be a properly convex domain of $\P(\R^n)$. Then $\Omega$ is $\COne$ if and only if $\Omega^\ast$ is strictly convex.
\end{lemma}
\begin{proof}
   Observe that the points on a projective line in $\P((\R^n)^\ast)$ can be interpreted in $\P(\R^n)$ as the set of projective hyperplanes containing a fixed codimension $2$ projective subspace. Therefore, a non-trivial segment in $\partial\Omega^\ast$ corresponds to a set of supporting hyperplanes of $\Omega$, all intersecting a common codimension $2$ projective subset $F$ of $\P(\R^n)$.

   Since $\Omega$ is contained in the two affine charts defined by the two hyperplanes corresponding to the extreme points of the segment, and since the hyperplane corresponding to a point in the interior of the segment only meets the closure of the intersection of these two affine charts at $F$, it follows that $\overline{\Omega}\cap F\neq\varnothing$. It is clear that $\Omega$ is not $\COne$ at any of the points of $\overline{\Omega}\cap F\subset\partial\Omega$.

   Conversely, if $\Omega$ is not $\COne$ at $x\in\partial\Omega$, there are two different supporting hyperplanes $H_1$ and $H_2$ at $x$. Therefore, the hyperplane defined by $x$ in $\P((\R^n)^\ast)$ meets $\partial\Omega^\ast$ at two distinct points corresponding to $H_1$ and $H_2$, so $\partial\Omega^\ast$ contains the non-trivial segment that lies in between, and $\Omega^\ast$ is not strictly convex.
\end{proof}

In fact, we can push further the correspondence established in the last lemma.

\begin{remark}
   The set of supporting hyperplanes of $\Omega$ at a boundary point $x$ is a face of the boundary of $\Omega^\ast$. This is because it is the intersection of $\partial\Omega^\ast$ with the supporting hyperplane defined by $x$.
\end{remark}

\subsection{Classification of automorphisms}

A very important and classical result in convex projective geometry is that, similarly to the picture in the hyperbolic space, there is a trichotomy for non-trivial isomorphisms based on their eigenvalues.

\begin{definition}
   Let $\Omega\subset\P(\R^n)$ be a properly convex domain, and $g\in\Aut(\Omega)$ a non-trivial automorphism. Define $\trans(g):=\inf_{x\in\Omega}d(x,gx)$. Either
   \begin{enumerate}
      \item $\trans(g) = 0$ and the infimum is achieved in $\Omega$: we say that $g$ is \emph{elliptic};
      \item $\trans(g) = 0$ and the infimum is not achieved in $\Omega$: we say that $g$ is \emph{parabolic};
      \item $\trans(g) > 0$: we say that $g$ is \emph{hyperbolic}.
   \end{enumerate}
\end{definition}

We can express the translation length of $g$ in terms of its eigenvalues.

\begin{lemma}[{{\cite[\Prop 2.1]{cooperLongTillmann2015}}}]\label{lm:transLengthIsRatioEV}
   Let $\Omega\subset\P(\R^n)$ be a properly convex domain, and $g\in\Aut(\Omega)$. Then $\trans(g)=\frac12\log\lambda_1(g)/\lambda_n(g)$, where $\lambda_1(g)\geq \dots \geq\lambda_n(g)$ denote the moduli of the eigenvalues of a lift of $g$ to $\GL(n,\R)$.
\end{lemma}

It follows that elliptic and parabolic automorphisms have eigenvalues of modulus $1$. We can discriminate between elliptic and parabolic isometries by being more precise.

\begin{lemma}[{{\cite[\Lm 2.2]{cooperLongTillmann2015}}}]
   Let $\Omega\subset\P(\R^n)$ be a properly convex domain, and $g\in\Aut(\Omega)$. Then $g$ is elliptic if and only if a lift of $g$ into $\SLpm(n,\R)$ is conjugated into $\O(n)$.
\end{lemma}

It is important to observe that unlike in hyperbolic geometry, where hyperbolic isometries always have a well-defined and unique axis in $\H^n$, this is not always the case in the convex projective world.

In general, the segments on which a hyperbolic isometry $g$ achieves minimal displacement will be called \emph{axes} of $g$. They do not have to be unique, nor do they have to be contained in the interior of $\Omega$.

\subsection{Vinberg hypersurfaces}

One last tool in convex projective geometry is the existence, for any sharp open convex cone $C$ in $\R^n$, of a foliation of $C$ by regular hypersurfaces that are invariant under any automorphism. This can be seen as a generalization to any properly convex domain of the hyperboloid model of hyperbolic space.

\begin{definition}
   Let $C$ be a sharp open convex cone, and $S$ be an hypersurface in $C$. We say that $S$ is a \emph{strictly convex analytic hypersurface} in $C$ if the connected component of $C-S$ that does not contain $0$ in its closure is strictly convex, and its boundary $S$ is defined by analytic equations in the canonical coordinates on $\R^n$.

   We say that $S$ is \emph{asymptotic} to $C$ if moreover there is no affine half-line contained in the connected component of $C-S$ containing $0$.
\end{definition}

There are two classical ways of producing such hypersurfaces: Vinberg hypersurfaces and affine spheres. We will focus here on Vinberg hypersurfaces, which are easier to define in our setting. For an introduction to affine spheres covering the topic of its applications to convex projective geometry, see the following survey of Loftin \cite{loftin2010}. In particular, Theorem 3 from this survey would be the result corresponding to Theorem \ref{thm:VinHypersurfaces} below.

Choose a volume form $\mathrm{d}\alpha$ on $(\R^n)^\ast$. The \emph{characteristic function} associated to a sharp convex cone $C$ is
\begin{equation*}
   f_C : \left\{\begin{aligned}
      C &\to \R \\
      x &\mapsto \int_{C^\ast} e^{\alpha(x)}\mathrm{d}\alpha
   \end{aligned}\right.
\end{equation*}
It is elementary to check that this defines an analytic function satisfying
\begin{equation}\label{eq:characFuncEquivariance}
   f_C(g x) = \abs{\det g}^{-1} f_{gC}(x)
\end{equation}
for any $g\in\GL(n,\R)$ and $x\in C$ (for instance, see \cite[\Lm 6.6]{goldman1988}). In particular, it is invariant under the subgroup of $\SLpm(n,\R)$ that preserves $C$. The level sets of $f_C$ define a foliation of $C$ by hypersurfaces.

\begin{theorem}[{{\cite{vinberg1963}}}]\label{thm:VinHypersurfaces}
   Let $C\subset\R^n$ be a sharp open convex cone. There exists a foliation of $C$ by strictly convex analytic hypersurfaces asymptotic to $C$, and invariant under the lift of $\Aut(\P(C))$ to $\SLpm(n,\R)$ preserving $C$. 
\end{theorem}

There are a number of useful applications of these Vinberg hypersurfaces.

\begin{proposition}[{{\cite{vinberg1963}}}]\label{prop:dualMap}
  Let $\Omega\subset\P(\R^n)$ be a properly convex domain. There exists a $\Aut(\Omega)$-equivariant diffeomorphism from $\Omega$ to $\Omega^\ast$.
\end{proposition}
\begin{proof}[Sketch of proof]
  Consider $S\subset\R^n$ a Vinberg hypersurface for $\Omega$. Let $\pi_S$ be the radial projection from $\Omega$ onto $S$ (meaning by rays emanating from $0\in\R^n$). For all $x\in\Omega$, $S$ has a unique well-defined (affine) tangent hyperplane $H_{\pi_S(x)}$ at the point $\pi_S(x)$. Its direction defines a point of $\Omega^\ast$. The resulting application from $\Omega$ to $\Omega^\ast$ is clearly infinitely differentiable and $\Aut(\Omega)$-equivariant. This is the map that we consider.

  See \cite{vinberg1963} for the rest of the proof, or \cite{goldman1988} for a proof in English.
\end{proof}

We can also use them in order to prove the existence of centers of mass for compact subsets.

\begin{proposition}[{{\cite[\Lm 4.2]{marquis2014}}}]\label{lm:centerOfMass}
  Let $\Omega\subset\P(\R^n)$ be a properly convex domain, and $K\subset\Omega$ be a compact subset. There exists a point $x_K\in\Omega$ such that any automorphism preserving $K$ fixes $x_K$.
\end{proposition}
\begin{proof}
  Let $S\subset\R^n$ be a Vinberg hypersurface for $\Omega$, and $K'$ be the convex hull of the projection of $K$ onto $S$. $K'$ is a compact convex subset of $\R^n$, so it has a center of mass. The corresponding point in $\Omega$ is $x_K$.
\end{proof}

Later, we will use the Vinberg hypersurfaces in order to smooth out domains.

\section{Elementary groups and cusps}

We now recall the definitions and properties of algebraic horospheres and elementary groups given in sections 3 and 4 of \cite{cooperLongTillmann2015}, as well as a smoothing lemma from \cite{cooperLongTillmann2015} which plays a very important role in our theory.

\subsection{Algebraic horospheres}

Consider a pair $(\xi,H)$ consisting of a point $\xi$ in the boundary of some properly convex domain $\Omega\subset\P(\R^n)$ and $H$ a supporting hyperplane of $\Omega$ at $\xi$. The horospheres of $\Omega$ based at $(\xi,H)$ are defined using the \emph{$(\xi,H)$-flow}.

\begin{definition}
   In a basis whose first vector is a lift of $\xi$, and whose $n-1$ first vectors span the lift of $H$, the \emph{$(\xi,H)$-flow} is given by the applications
   \begin{equation*}
      \phi^{(\xi,H)}_t = \begin{pmatrix}
         1 & 0 & t \\
         0 & I_{n-2} & 0 \\
         0 & 0 & 1
      \end{pmatrix}
   \end{equation*}
   for $t\in\R$.

   If $S^\Omega_\xi$ is the union of all segments in $\partial\Omega$ containing $\xi$, the \emph{algebraic horospheres of $\Omega$ centered at $(\xi,H)$} are the images of $\partial\Omega-S^\Omega_\xi$ under $\phi^{(\xi,H)}_t$ for $t>0$. If $\Omega$ and $(\xi,H)$ are understood from the context, we may write $\horo_t$ for the image of $\partial\Omega-S^\Omega_\xi$ under $\phi_t^{(\xi,H)}$.
\end{definition}

So $\Omega$ is foliated by its algebraic horospheres centered at $(\xi,H)$. It turns out we can say more.

\begin{remark}\label{rmk:ConeFoliatedByAlgHoro}
   If we denote by $C^\Omega_\xi$ the intersection of all half-spaces containing $\Omega$ whose boundary hyperplane is supporting to $\Omega$ at $\xi$, we see that it is entirely foliated by the translates of $\partial\Omega-S^\Omega_\xi$ by the $(\xi,H)$-flow.
\end{remark}

Notice that in the affine chart $\P(\R^n)-H$, $\phi_t^{(\xi,H)}$ acts by translating in the direction of $\xi$ by a linear function of $t$. It follows that the $(\xi,H)$-algebraic horospheres are the translates of $\partial\Omega-S^\Omega_\xi$ in the direction of $\xi$. This choice of affine chart allowing for an easy description of the algebraic horospheres is what \cite{cooperLongTillmann2015} call \emph{parabolic coordinates} for $(\xi,H)$.

It is not difficult to tell whether an automorphism of $\Omega$ preserves algebraic horospheres.

\begin{proposition}[{{\cite[\Prop 3.2]{cooperLongTillmann2015}}}]\label{prop:autPreservesAlgHoroNSC}
   Let $\Omega\subset\P(\R^n)$ be a properly convex domain, $\xi\in\partial\Omega$ and $H$ a supporting hyperplane of $\Omega$ at $\xi$. An automorphism $g\in\Aut(\Omega)$ preserves the foliation of $\Omega$ by its algebraic horospheres based at $(\xi,H)$ if and only if $g(\xi,H)=(\xi,H)$. In this case, $g$ sends the horosphere $\horo_t$ to $\horo_{t\tau_{(\xi,H)}(g)}$, where $\tau(g)$ is the quotient of the eigenvalues of $g$ associated to $\xi$ by that associated to $H$. 
\end{proposition}

In particular, observe that parabolic and elliptic elements preserving the center of an algebraic horosphere always preserve the horospheres, whereas hyperbolic elements fixing the center of the horospheres preserve them if and only if they have the same eigenvalue on $\xi$ and $H$. What is meant here by having the same eigenvalue on $\xi$ and $H$ is that, in a basis whose first vector is $\xi$, and whose first $n-1$ vectors span $H$, the automorphism is of the form
\begin{equation*}
   \begin{bmatrix}
      \lambda & \ast & \ast \\
      0 & \ast & \ast \\
      0 & 0 & \lambda
   \end{bmatrix}
\end{equation*}
where $\lambda\in\R$ is the eigenvalue of $\xi$ and $H$. Observe that having the same eigenvalue does not depend on the choice of a representative in $\GL(n,\R)$.

We also observe that the subgroup of $\Aut(\Omega)$ preserving algebraic horospheres commutes with the associated flow.

\begin{lemma}\label{lm:autThatPreservesAlgHoroCommutesWithFlow}
   Let $\Omega\subset\P(\R^n)$ be a properly convex domain, $\xi\in\partial\Omega$ and $H$ a supporting hyperplane of $\Omega$ at $\xi$. If $g\in\Aut(\Omega)$ preserves the algebraic horospheres of $\Omega$ centered at $(\xi,H)$, then $g$ commutes with the $(\xi,H)$-flow.
\end{lemma}
\begin{proof}
   By Proposition \ref{prop:autPreservesAlgHoroNSC}, $g$ has the same eigenvalue $\lambda$ for both $\xi$ and $H$ (seen as eigenvectors). Therefore, $g$ is of the form
   \begin{equation*}
      \begin{bmatrix}
         \lambda & \ast & \ast \\
         0 & \ast & \ast \\
         0 & 0 & \lambda
      \end{bmatrix}
   \end{equation*}
   in a basis $e_1,\dots,e_n$ where $e_1$ is a lift of $\xi$, and $\gen{e_1,\dots,e_{n-1}}$ is the hyperplane lifting $H$.

   Checking that such a matrix commutes with the $(\xi,H)$-flow is then an elementary computation.
\end{proof}

\subsection{Elementary groups in strictly convex domains}\label{ssec:elemGps}

The following is a summary of the main results of Section 4 in \cite{cooperLongTillmann2015}. They were written in the case of torsion-free subgroups for simplicity, but the same proofs easily adapt to the general case.

\begin{definition}
   A group $G$ of automorphisms of a properly convex domain $\Omega\subset\P(\R^n)$ is \emph{elementary} if it admits a globally fixed point $x$ in the closure of $\Omega$. We say that $G$ is
   \begin{enumerate}
      \item \emph{elementary of the elliptic type} if there is a fixed point in $\Omega$;
      \item \emph{elementary of the parabolic type} if there is a unique fixed point in $\partial\Omega$, none in $\Omega$, and $G$ contains no hyperbolic elements;
      \item \emph{elementary of the hyperbolic type} if there are two fixed points in $\partial\Omega$, and none in $\Omega$.
   \end{enumerate}
   In the last two cases, we say that $G$ is \emph{doubly elementary} if there exists $\xi\in\partial\Omega$ and $H$ a supporting hyperplane of $\Omega$ at $\xi$ such that $G\cdot (\xi,H)=(\xi,H)$.
\end{definition}

Observe that a priori, nothing keeps an elementary group from preserving more than $2$ points in the boundary of a properly convex domain.

\begin{example}
   If $\Omega$ is a $k$-dimensional simplex in $\P(\R^{k+1})$, then the group of diagonal matrices in the basis given by the vertices of the simplex is elementary fixing $k+1$ distinct points, but is not elliptic.
\end{example}

However, in the strictly convex case, elementary groups satisfy this trichotomy.

\begin{lemma}
   Let $\Omega\subset\P(\R^n)$ be a strictly convex domain, and $G<\Aut(\Omega)$ be an elementary group. If $G$ fixes three distinct points in $\partial\Omega$, then $G$ is elliptic.
\end{lemma}
\begin{proof}
   Suppose $G$ fixes $3$ distinct points $x,y,z\in\partial\Omega$. Since $\Omega$ is strictly convex, these three points cannot be on the same projective line. Let $\Pi$ be the projective plane determined by $x$, $y$ and $z$. $\Pi$ is clearly $G$-invariant, and intersects $\Omega$ in a strictly convex domain $\omega$ of $\Pi$. Moreover, the restriction of $G$ to $\SLpm(\Span\Pi)$ is a diagonal group with weights $\lambda_x$, $\lambda_y$ and $\lambda_z$ associated to $x$, $y$ and $z$ respectively.

   Let $g\in G$ be a non-trivial element. Suppose $\{\lambda_x(g),\lambda_y(g),\lambda_z(g)\}$ has cardinal larger than $1$. If it has cardinal $2$, then all the properly convex domains preserved by $g$ are triangles (see \cite[\Prop 2.10]{marquis2012surface}), which contradicts the strict convexity of $\omega$. Similarly, if it has cardinal $3$, any $g$-invariant domain containing $x$, $y$ and $z$ in its boundary would contain at least the segments $[x,y]$ and $[y,z]$ in its boundary if we suppose $\lambda_x(g)>\lambda_y(g) >\lambda_z(g)$, contradicting the strict convexity of $\omega$. This is because on the complement in $\Pi$ of the $2$ projective lines going through $x$ and $y$, and through $y$ and $z$, any forward (\resp backward) $g$-orbit limits to $x$ (\resp $z$), and because the two affine charts determined by these two lines are preserved by $g$, so the proper convexity of $\omega$ implies it does not intersect these two lines. For more details, see \cite[\Prop 2.9]{marquis2012surface}.

   It follows that $g$ acts as the identity on $\Pi$, so it fixes any point in the interior of $\omega$. As this is true of any element of $G$, $G$ is elementary of the elliptic type.
\end{proof}

We now summarize the description of each type of elementary subgroups given by Cooper, Long and Tillmann.

\begin{proposition}
   Let $\Omega$ be a strictly convex domain, and $\Gamma$ be a subgroup of automorphisms of $\Omega$. Then $\Gamma$ is elementary of the elliptic type if and only if a lift of $\Gamma$ into $\SLpm(n,\R)$ is conjugated into $\O(n)$.
\end{proposition}
\begin{proof}
   If $G$ is conjugated into $\O(n)$ up to lifting it to $\SLpm(n,\R)$, $G$ has compact closure in $\Aut(\Omega)$, so by \cite[\Lm 4.3]{cooperLongTillmann2015}, $G$ fixes a point in $\Omega$, \ie $G$ is elementary of the elliptic type.

   Conversely, if $G$ is elementary of the elliptic type, consider a point $x\in\Omega$ that is fixed by $G$. Therefore, $G$ also preserves the hyperplane $H$ dual to $x$ (see Proposition \ref{prop:dualMap}), and the affine chart defined by $H$. Since $G$ preserves $\Omega$, it must also preserve its John ellipsoid in the affine chart defined by $H$. Since the Hilbert geometry of an ellipsoid is hyperbolic geometry, and the result holds for elementary elliptic subgroups of $\PO(n-1,1)$, we see that a lift of $G$ to $\SLpm(n,\R)$ is conjugated into $\O(n)$.
\end{proof}

\begin{proposition}\label{prop:CLTHypElemSgpSummary}
   Let $\Omega$ be a strictly convex domain, and $\Gamma$ an infinite discrete torsion-free subgroup of $\Aut(\Omega)$. The following are equivalent:
   \begin{enumerate}
      \item $\Gamma$ is doubly elementary of the hyperbolic type;
      \item $\Gamma$ is elementary of the hyperbolic type;
      \item $\Gamma$ is cyclic and has no parabolic element.
      \item $\Gamma$ is virtually nilpotent and has no parabolic element;
   \end{enumerate}
\end{proposition}
\begin{proof}
   Clearly, being doubly elementary of the hyperbolic type implies being elementary of the hyperbolic type. If $\Gamma$ is elementary of the hyperbolic type, it does not have any parabolic elements since these fix at most one point in $\overline{\Omega}$. Moreover, $\Gamma$ is cyclic since it is discrete and torsion-free. Clearly, being cyclic implies being virtually nilpotent. Finally, if $\Gamma$ is virtually nilpotent and has no parabolic elements, it is elementary of the hyperbolic type by \cite[\Prop 4.14 \& 4.13]{cooperLongTillmann2015}.
\end{proof}

\begin{proposition}[{{\cite[sec. 4]{cooperLongTillmann2015}}}] \label{prop:CLTParabElemSgpSummary}
   Let $\Omega$ be a strictly convex domain, and $\Gamma$ an unbounded subgroup of $\Aut(\Omega)$. The following are equivalent
   \begin{enumerate}
      \item $\Gamma$ is doubly elementary of the parabolic type;
      \item $\Gamma$ is elementary of the parabolic type;
      \item $\Gamma$ is weakly unipotent;
      \item $\Gamma$ has no hyperbolic elements.
   \end{enumerate}
   If moreover $\Gamma$ is discrete, then it is virtually nilpotent.
\end{proposition}

\begin{proof}
   Clearly, being doubly elementary of the parabolic type implies being elementary of the parabolic type. If $\Gamma$ is an elementary group of the parabolic type, it does not have any hyperbolic elements, so it is weakly unipotent by Lemma \ref{lm:transLengthIsRatioEV}. Finally, if $\Gamma$ has no hyperbolic elements, it is doubly elementary by \cite[\Cor 4.7]{cooperLongTillmann2015}, necessarily of the parabolic type, since it has no loxodromic element, and any non-trivial parabolic element in $\Gamma$ would fix two distinct points in $\overline{\Omega}$.

   If moreover $\Gamma$ is discrete, it is virtually nilpotent by \cite[\Prop 4.11]{cooperLongTillmann2015}.
\end{proof}

In the general properly convex case, we have.

\begin{proposition}\label{prop:solvGpsAreElem}
   Let $\Omega\subset\P(\R^n)$ be a properly convex domain, and $R<\Aut(\Omega)$ be a solvable Zariski connected group. Then $R$ is elementary.
\end{proposition}
\begin{proof}
    We prove the result by induction on the dimension. The result is clear for $n=1$. Suppose $n\geq 2$. By the Lie--Kolchin theorem, $R$ either preserves a projective line coming from an eigenvector in $\P(\C^n)-\P(\R^n)$, or it preserves a point in $\P(\R^n)$. Observe that if $R$ preserves a line, then it cannot meet $\overline{\Omega}$, since $R$ would act by rotations on this line, which would contradict the proper convexity of $\Omega$.

    Hence we may suppose that $R$ preserves a projective subspace $\P(E)$ of dimension $0$ or $1$ which does not meet $\overline{\Omega}$. Since $E\cap\overline{\Omega}=\varnothing$, $\Omega$ projects to a properly convex domain in $\P(\R^n/E)$, under the projection $\P(\R^n)-\P(E)\to\P(\R^n/E)$. By the induction hypothesis, we find a fixed point in its closure, which lifts to a $R$-invariant projective subspace $\P(F)$ which intersects $\overline{\Omega}$ and contains $\P(E)$ as a hyperplane.

   Therefore, in $\P(F)$, $R$ preserves the affine chart determined by $\P(E)$, as well as the bounded convex subset (in this affine chart) $\overline{\Omega}\cap\P(F)$. Since $R$ is solvable, it is amenable, so any affine action on a compact convex set has a fixed point in this compact set. This gives a fixed point in $\overline{\Omega}$.
\end{proof}

As a corollary, we get the following generalization of \cite[\Prop 4.14]{cooperLongTillmann2015}.

\begin{corollary}\label{cor:strictlyConvDomainsElemIFFVirtSolv}
   Let $\Omega\subset\P(\R^n)$ be a strictly convex domain, and $R<\Aut(\Omega)$ be a virtually solvable torsion-free group. Then $R$ is elementary.
\end{corollary}
\begin{proof}
   Consider $R'$ a finite index solvable Zariski connected subgroup of $R$. $R'$ is elementary by Proposition \ref{prop:solvGpsAreElem}.

   If $R'$ is elementary of the elliptic type, we can find $x\in\Omega$ such that $R'\cdot x = x$. Since $R'$ is a finite index subgroup of $R$, $R\cdot x\subset\Omega$ is a finite subset, and is therefore bounded. It is clearly $R$-invariant, so Lemma \ref{lm:centerOfMass} gives a point $y\in\Omega$ that is fixed by $R$. Therefore, $R$ is elementary of the elliptic type.

   Otherwise, $R'$ fixes a point $x\in\partial\Omega$. For any $g\in R-R'$, $g^n\in R'$ for some $n>1$, so $g^n x = x$. Since $g$ is not elliptic, it follows that $g$ also fixes $x$, so $R$ is elementary.
\end{proof}

Lastly, we give the definition of cusps we will be using throughout this article. Notice that it is different from the more classical definition found in the introduction.

\begin{definition}\label{def:cusps}
   For $\Omega$ a strictly convex (\resp round) domain, and $\Gamma<\Aut(\Omega)$ infinite and discrete, we say that $\Omega/\Gamma$ is a \emph{strictly convex (\resp round) cusp} if $\Gamma$ fixes a unique point in $\partial\Omega$.
\end{definition}

The following is a direct corollary of Proposition \ref{prop:CLTParabElemSgpSummary}.

\begin{lemma}
   Let $\Gamma$ be a discrete infinite subgroup of automorphisms of a strictly convex domain $\Omega$. Then $\Omega/\Gamma$ is a cusp if and only if the action of $\Gamma$ is doubly elementary of the parabolic type.
\end{lemma}

\subsection{A smoothing lemma for elementary groups of the parabolic type}

We now give a generalization of a smoothing lemma by Cooper--Long--Tillmann \cite{cooperLongTillmann2015} allowing to smooth out domains preserved by elementary subgroups of the parabolic type. This is one of the main ingredients of many of the results of this work. The original version was only stated in the case where the elementary group is unipotent, but the same proof yields a slightly more general statement, which will later enable us to build a generalized cusp with non-virtually nilpotent holonomy (see paragraph \ref{par:genCuspWithSolvHolo}).

\begin{lemma}[{{\cite[\Prop 5.8]{cooperLongTillmann2015}}}]\label{lm:smoothing}
   Suppose $\Gamma < \PGL(n,\R)$ preserves a properly convex domain $\Omega\subset\P(\R^n)$, as well as a hyperplane $H$ supporting $\Omega$ with trivial weight. Then $\Omega$ is foliated by $\Gamma$-invariant hypersurfaces that are the graphs of strictly convex analytic functions in the affine chart $\P(\R^n)-H$.

   The interior $\Omega'$ of the convex hull of one of these hypersurfaces is such that $\partial\Omega'$ meets $\partial\Omega$ at $\partial\Omega\cap H$ exactly and $\partial\Omega'-H$ is $\COne$ and contains no non-trivial segments. If moreover $\partial\Omega\cap H$ is a single point, $\Omega'$ is strictly convex, and if this point is $\COne$ in $\partial\Omega$, then $\Omega'$ is round.
\end{lemma}

\begin{proof}
   Choose a convex cone $C$ lifting $\Omega$ to $\R^n$, and a lift $\Gamma'<\SLpm(n,\R)$ of $\Gamma$ preserving $C$. We identify $\R^n$ with the affine chart of $\P(\R^{n+1})$ determined by the hyperplane $H_0$ defined by the equation $x_{n+1}=0$, and let $o=[0:\dots :0:1]$ be the point identified with the origin of $\R^n$. This makes $C$ a properly convex domain of $\P(\R^{n+1})$ since $\Omega$ is properly convex in $H_0$ and $o\not\in H_0$. We realize $\Gamma'$ as a subgroup of $\Aut(C)$ via the embedding
   \begin{equation*}
      \gamma\mapsto\begin{bmatrix}
         \gamma & 0 \\
         0 & 1
      \end{bmatrix}
   \end{equation*}

   By Theorem \ref{thm:VinHypersurfaces}, any Vinberg hypersurface $S\subset C$ is invariant under the action of $\Gamma'$. Let $H_1$ be the projective hyperplane containing $H\subset H_0$ and $o$. Observe that all elements of $\Gamma'$ fix both $H_0$ and $H_1$, with eigenvalue $1$. Therefore, $\Gamma'$ acts on $\P(\R^{n+1}/\Span H)\simeq\P(\R^2)$ by fixing the two points corresponding to $H_0$ and $H_1$, with eigenvalue $1$, so $\Gamma'$ acts trivially on $\P(\R^{n+1}/\Span H)$. It follows that $\Gamma'$ fixes any hyperplane in $\P(\R^{n+1})$ containing $H$.

   Choose a hyperplane $H_2$ containing $H$ and meeting the interior of $C$. We claim that $\Omega_{H_2}:=\IntRel H_2\cap \Conv S$ is a properly convex domain of $H_2$ whose boundary is the graph of a strictly convex analytic function in $H_2-H$. Moreover, if $H\cap\overline{\Omega}$ is reduced to a point, then the domain is strictly convex, and if this point is a $\COne$ point of $\partial\Omega$, then the domain is round. Assuming this holds, the lemma follows because $C$, and hence $S$ is foliated by hyperplanes containing $H$ and meeting with $C$, so that the projection into $\Omega$ of $\partial\Omega_{H_2}$ for varying $H_2$ defines a foliation of $\Omega$ by hypersurfaces as required.

   The first claim comes from the fact that $S$ is the graph of a strictly convex analytic function in the affine chart $\P(\R^{n+1})-H_0$, hence its intersection with $H_2$ is also the graph of a strictly convex analytic function in the affine chart $H_2-H_2\cap H_0=H_2-H$ of $H_2$. The second claim follows directly from the first. For the third one, observe that a supporting hyperplane $H'$ of $\Omega_{H_2}$ in $H_2$ at the point $H\cap\overline{\Omega}$ defines a supporting hyperplane $\P(\Span H'\oplus \Span o)$ to $C$ at the point $H\cap\overline{\Omega}$ whose intersection with $H_0$ is a supporting hyperplane of $\Omega$ at $H\cap\overline{\Omega}$. Since this point is $\COne$ in $\partial\Omega$, $H'$ is uniquely determined and must be $H_1\cap H_2=H$, so $\Omega_{H_2}$ is $\COne$.
\end{proof}

\section{Properness and asymptotics of linear combinations of multivariate functions}

\subsection{Order and domination}

Let $k\geq 2$ and $X\subset \R^k$ be a subset with the induced topology. In most places, we will set $X=\Gamma\subset \GL(n,\R)\subset \M(n,\R)\simeq\R^{n^2}$ discrete, but the case where $X\simeq \R^k$ is the Lie algebra of a nilpotent Lie group also appears. We will denote by $\Cont(X)$ the set of continuous functions from $X\to\R$. We want to study the asymptotics at infinity in $X$ of finite sets of functions of $\Cont(X)$, \ie the asymptotics along any sequence in $X$ that leaves all compacts.

We will make extensive use of the Landau notation at infinity. For this reason, unless specified otherwise, all Landau symbols should be interpreted at infinity. We recall the $\Theta$ notation, which is slightly less common.

\begin{definition}
   Let $f,g\in\Cont(X)$. We say that \emph{$f$ and $g$ the same order}, and we write $f = \Theta(g)$ when $f = O(g)$ and $g = O(f)$.

   Following the convention that $\Theta(g)$ denotes the set of all functions having the same order as $g$, we may write as well $\Theta(f)=\Theta(g)$.
\end{definition}

This means that there is a constant $A>1$ and a compact subset $K\subset X$ such that for all $x\in X-K$
\begin{equation*}
   \frac1A \abs{f(x)} \leq \abs{g(x)} \leq A \abs{f(x)}
\end{equation*}
In other words, one of $\abs{f}$ or $\abs{g}$ goes to $0$ or infinity along a divergent sequence $x_n\to\infty$ in $X$ if and only if the other does so, and when this happens, they do so at the same rate.

We will mainly study the order of sets of functions in what follows.

\begin{definition}
   The order of a finite set of functions $S\subset\Cont(X)$ is the order of $\max_{f\in S}\abs{f}$. We will write
   \begin{equation*}
      \Theta(S) := \Theta(\max_{f\in S}\abs{f})
   \end{equation*}
   Similarly, $g\in\Cont(X)$ is $o(S)$ (\resp $O(S)$) if and only if it is $o(\max_{f\in S}\abs{f})$ (resp. $O(\max_{f\in S}\abs{f})$).
\end{definition}

It is easy to verify that the order of a finite set of functions is equal to the order of the sum of the absolute values of these functions.

\begin{lemma}
   Let $S\subset\Cont(X)$ be a finite set. Then
   \begin{equation*}
      \Theta(S) = \Theta\left(\sum_{f\in S}\abs{f}\right)
   \end{equation*}
\end{lemma}

Lastly, in a given set of functions, some may not contribute to the order of the whole set. We formalize this with the notion of \emph{domination}.

\begin{definition}
   Given a finite set of continuous functions $S\subset \Cont(X)$, we say that
   \begin{itemize}
      \item $s\in S$ is \emph{dominated} by a subset $T$ of $S$ if $s = o(T)$.
      \item $s\in S$ is \emph{dominating} if it is not dominated by $S$.
   \end{itemize}
   We denote by $S^+$ the subset of dominating functions in $S$.
\end{definition}

\begin{lemma}
   If $s\in S-S^+$, then $s$ is dominated by $S^+$.
\end{lemma}

\begin{proof}
   By definition of dominated elements of $S$, we have, for all $s\in S-S^+$:
   \begin{equation*}
      s = o(S)
   \end{equation*}

   Therefore
   \begin{equation*}
      \Order(S) = \Order(S^+ \cup (S-S^+)) = \Order(S^+)
   \end{equation*}
   so that
   \begin{equation*}
      s = o(S) = o(S^+)
   \end{equation*}
   for all $s\in S-S^+$.

\end{proof}

\subsection{Generic s-proper linear combinations}

We now need to consider “proper” functions on $X\subset\R^k$. Our concept of properness needs to allow for continuous extension to larger subsets. That is, if $Y\supset X$, continuous proper functions on $X$ should extend to continuous proper functions on $Y$. This is because our main application case will be when $X$ is an unbounded subgroup of $\SLpm(n,\R)$, seen as a subset of the set of $n$ by $n$ matrices, identified with $\R^{n^2}$. We need to be able to say that if $X$ is a subgroup of larger Lie group $Y$, then existence of some proper functions on $X$ gives proper functions on $Y$ as well.

However, the classical topological notion of properness is clearly not suited to our main application case, since if $X$ is “too disconnected”, some continuous proper functions on $X$ would be very far from being proper on a larger connected subset. For instance, if $X := \Z^2\subset Y:=\R^2$,

the function $f:(a,b)\in\Z^2\mapsto (-1)^{a+b}\allowbreak \exp(a^2+b^2)$ defines a continuous proper function on $X$, but it cannot be extended to a proper continuous function on $Y$ as it changes sign on the complement of any compact subset, by Lemma \ref{lm:properIsBoundedBelow} below.

Instead, since we are interested in extending proper maps continuously, we may consider the realm of continuous functions on $\R^k$, where a stronger condition holds.

\begin{lemma}\label{lm:properIsBoundedBelow}
   A function $f\in\Cont(\R^k)$ is proper if and only if either $f\to_{\infty}+\infty$, or $f\to_{\infty} -\infty$.
\end{lemma}
\begin{proof}
   Let $A > 0$. Suppose one can find arbitrarily large $x,y\in \R^k$ such that $f(x)>A$ and $f(y)<-A$. Since $f$ is proper, outside of a large enough ball, any $z$ satisfy $\abs{f(z)}> A$. However since $\R^k$ minus a ball is connected, we can find a continuous path from $x$ to $y$ that avoids this large ball. Since $f$ is continuous, it must take the value $0$ along that path. This is a contradiction.
   
   The conclusion follows.
\end{proof}

\begin{definition}\label{def:sProper}
   We say that $f\in\Cont(X)$ is \emph{s-proper} if either $f\to_{\infty}+\infty$ or $f\to_{\infty}-\infty$.
\end{definition}

We may now start to investigate s-proper linear combinations of continuous functions on $X$.

\begin{lemma}\label{lm:setOfsPLCIsConvex}
   Let $f_1,\dots,f_s\in\Cont(X)$ be continuous functions such that
   \begin{equation*}
      \Lambda = \{(\lambda_1,\dots,\lambda_s)\in\R^s,\text{$\sum\lambda_i f_i$ is s-proper and bounded below}\}
   \end{equation*}
   is not empty.

   Then $\Lambda$ is convex and the linear combinations corresponding to points in the relative interior of $\Lambda$ all have the same order.
\end{lemma}

\begin{proof}
   Let $(\lambda_i)_i, (\mu_i)_i\in\Lambda$, and $t\in (0;1)$. Observe that
   \begin{align*}
      \sum(t\lambda_i + (1-t)\mu_i)f_i(x) &= t\sum\lambda_i f_i(x) + (1-t)\sum\mu_i f_i(x) \\
      &\geq \min\{t,1-t\}(\sum\lambda_if_i(x) + \sum\mu_if_i(x))
   \end{align*}
   when $x\in X$ is large enough that both $\sum\lambda_if_i(x)$ and $\sum\mu_if_i(x)$ are positive. Therefore $\Lambda$ is convex.

   We also have with the same condition on $x$:
   \begin{equation*}
      t\sum\lambda_if_i(x) + (1-t)\sum\mu_if_i(x) \leq \sum\lambda_if_i(x) + \sum\mu_if_i(x)
   \end{equation*}
   so we see that the linear combination associated to $(t\lambda_i + (1-t)\mu_i)_i$ has the same order as $\sum\lambda_if_i + \sum\mu_if_i$, which does not depend on the choice of $t\in(0;1)$.

   Therefore, since any two points in the relative interior of $\Lambda$ lie in the interior of a segment contained in $\Lambda$, all the points in the interior of $\Lambda$ are associated to s-proper bounded below linear combinations having the same order.
\end{proof}

We now introduce the main notion of this section.

\begin{definition}\label{def:genPLC}
   Keeping the notations of Lemma \ref{lm:setOfsPLCIsConvex}, if $\Lambda$ has non-empty interior in $\R^s$, s-proper linear combinations of $f_1,\dots, f_s$ that are associated to points in the interior of $\Lambda$ or $-\Lambda$ are said to be \emph{generic}.
\end{definition}

A s-proper linear combinations of continuous functions on $X$ is generic if we can perturb each of the coefficients of the linear combination independently of the others, while the resulting linear combinations remain s-proper. 

\begin{proposition}\label{prop:genPLCOrder}
   Let $f_1,\dots,f_s\in\Cont(X)$ be continuous functions that admit a generic s-proper linear combination $f$. Then $f$ has the same order as $\{f_1,\dots,f_s\}$.

\end{proposition}
\begin{proof}
   Write $f=\sum_i\lambda_if_i$ where $\lambda_1,\dots,\lambda_s\in\R$. We have for all $x\in X$:
   \begin{align*}
      \abs{f(x)}&\leq \sum\abs{\lambda_i}\abs{f_i(x)}\\
      &\leq s\max\abs{\lambda_i}\max\abs{f_i(x)}
   \end{align*}

   For the lower bound, we may suppose up to changing $f$ into $-f$ that $f$ is s-proper and bounded below. The genericity of $f$ implies the existence of $\epsilon>0$ such that any $(\mu_1,\dots,\mu_s)$ satisfying $\abs{\mu_i-\lambda_i}=\epsilon$ for all $i$ still lies in the interior of $\Lambda$ (as defined in Lemma \ref{lm:setOfsPLCIsConvex}). Since this is a finite subset of $\Lambda$, Lemma \ref{lm:setOfsPLCIsConvex} implies that there is a compact subset $K$ of $X$ and a uniform constant $A$ such that
   \begin{equation*}
      \frac1 A \abs{f(x)} \leq \sum_i\mu_if_i(x)\leq A \abs{f(x)}
   \end{equation*}
   for all $(\mu_1,\dots,\mu_s)$ as above, and $x\not\in K$. Up to enlarging $K$, we may suppose that $f$ is non-negative on the complement of $K$. Therefore, if we define $\epsilon_i(x)$ to be $\epsilon$ times the sign of $f_i(x)$ for any $x\in X$, we have
   \begin{equation*}
      f(x) - \epsilon\sum_i\abs{f_i(x)} = \sum_i(\lambda_i -\epsilon_i(x))f_i(x) \geq \frac1 A f(x) \geq 0
   \end{equation*}
   for all $x\not\in K$, from which follows
   \begin{equation*}
      f(x)\geq \epsilon\sum_i\abs{f_i(x)}\geq \epsilon\max_i\abs{f_i(x)}
   \end{equation*}
   which yields the result.

\end{proof}

Notice that the genericity hypothesis in Proposition \ref{prop:genPLCOrder} cannot be dropped, as demonstrated in the example below.

\begin{example}
   The polynomial $P(X,Y) = X^4 - 2X^2Y^2 + Y^4 + Y^2$ is s-proper and bounded below on $\R^2$. Indeed, whenever $\abs{X^2-Y^2}\geq \frac12 X^2$,
   \begin{equation*}
      P(X,Y) = (X^2-Y^2)^2 + Y^2\geq\frac14 X^4 + Y^2
   \end{equation*}
   and if $\abs{X^2-Y^2}\leq\frac12 X^2$, then $Y^2\geq\frac12 X^2$, so that
   \begin{equation*}
      P(X,Y)\geq Y^2 = \frac23 Y^2 + \frac13 Y^2\geq \frac13 (X^2 + Y^2)
   \end{equation*}

   It follows that $P$ is a s-proper linear combination of the polynomials $X^4$, $Y^4$, $X^2Y^2$ and $Y^2$ on $\R^2$. However, observe that $P$ does not have the same order as $\{X^4,X^2Y^2,Y^4,Y^2\}$ along the sequence $(X_n,Y_n) = (n,n)$ for $n\in\N$.

   It is easy to check directly that $P$ is not generic by observing that for any $\epsilon>0$, $P-\epsilon X^2 Y^2$ is not s-proper since along the same sequence as above, it goes to $-\infty$, while it goes to $+\infty$ along the sequence $(X_n,Y_n)=(0,n)$ for instance.
\end{example}

\subsection{Invariant domains and s-properness}

We start by describing the notations that we will use in the rest of this article. We will denote by $(e_1,\dots,e_n)$ a basis of $\R^n$ (the canonical basis if no precision is made), $(e_1^\ast,\dots,e_n^\ast)$ the dual basis of $(\R^n)^\ast$, and $(e_{ij})_{1\leq i,j\leq n}$ the associated basis of $\M(n,\R)$.

For a vector $x$ in $\R^n$ or $(\R^n)^\ast$, we denote by $x_i$ its entries (in the canonical basis unless stated otherwise). Similarly, if $\gamma\in\M(n,\R)$, we denote by $\gamma_{ij}$ its $(i,j)$-th entry with respect to the canonical basis, or another specified basis.

In order to denote entries or linear combinations of these entries for a whole group $\Gamma<\GL(n,\R)$, we will use $\mapsto$. For instance, the $(i,j)$-th entry of $\Gamma$ would be $\gamma\in\Gamma\mapsto \gamma_{ij}$, which we will often simplify to $\gamma\mapsto \gamma_{ij}$ to lighten the notation.

In the following, we will have to deal with s-proper linear combinations of the entries of subgroups $\Gamma$ of $\SLpm(n,\R)$, seen as a subset of $\M(n,\R)\simeq\R^{n^2}$. We stress that we do not make any assumption on the topology of $\Gamma$, so as to allow for discrete groups, as well as Lie groups. Certain linear combinations of the entries of $\Gamma$ arise naturally in our geometric setting.

\begin{definition}
   Let $\Gamma$ be a unbounded subgroup of $\SLpm(n,\R)$. A \emph{matrix coefficient} of $\Gamma$ is a linear combination of the entries of $\Gamma$ of the form
   \begin{equation*}
      \gamma\mapsto\alpha(\gamma\cdot x) = \sum_{ij}\alpha_ix_j\gamma_{ij}
   \end{equation*}
   for some $(\alpha,x)\in\R^n\times(\R^n)^\ast$.
\end{definition}

A s-proper matrix coefficient of $\Gamma<\SLpm(n,\R)$ might be generic in two different ways a priori: as a s-proper linear combination, or as a matrix coefficient (\ie we can choose $(x,\alpha)$ in an open subset of $\R^n\times(\R^n)^\ast$). Since the dimension of the set of matrix coefficients is only $2n$, and is smaller than the dimension $n^2$ of the set of linear combinations, it may seem at first that we can't infer one from the other. This is not the case as the following lemma explains.

\begin{lemma}\label{lm:genPGLCAreGenPLC}
   Let $\Gamma$ be an unbounded subgroup of $\SLpm(n,\R)$. If there are two open subsets $U$ and $V$ of $\R^n$ and $(\R^n)^\ast$ respectively such that for all $(x,\alpha)\in U\times V$, the matrix coefficients $\gamma\mapsto\alpha(\gamma\cdot x)$ are s-proper, then they are generic s-proper linear combinations of the entries of $\Gamma$.
\end{lemma}
\begin{proof}
   Suppose there are $U$ and $V$ as above, and consider
   \begin{equation*}
      \mathcal{L}:\left\{\begin{aligned}
         U\times V &\to \M(n,\R)\simeq\R^{n^2}\\
         (x,\alpha) &\mapsto (x_i\alpha_j)_{1\leq i,j\leq n}
      \end{aligned}\right.
   \end{equation*}
   By Lemma \ref{lm:setOfsPLCIsConvex}, it is enough to prove that $\Conv\mathcal{L}(U,V)$ has non-empty interior, since $\mathcal{L}(U,V)$ forms a subset of the set of s-proper linear combinations of the entries of $\Gamma$. This amounts to finding $n^2$ elements of $U\times V$ whose image under $\mathcal{L}$ span $\R^{n^2}$.

   Choose $(x,\alpha)\in U\times V$, and $\epsilon >0$ such that $B(x,2\epsilon)\times B(\alpha,2\epsilon)\subset U\times V$. For all $i,j\in\{1,\dots,n\}$, define $x^i = x + \epsilon e_i$ and $\alpha^j = \alpha + \epsilon e_j^\ast$. We compute
   \begin{equation*}
      \mathcal{L}(x^i,\alpha^j) = \mathcal{L}(x,\alpha) + \epsilon\begin{pmatrix}&0&\\&\vdots&\\\alpha_1 & \cdots & \alpha_n\\&\vdots&\\&0&\end{pmatrix} + \epsilon\begin{pmatrix}&&x_1&&\\0&\cdots&\vdots&\cdots&0\\&&x_n&&\end{pmatrix} + \epsilon^2 e_{ij}
   \end{equation*}
   where $\alpha$ is on row $i$ and $x$ on column $j$.

   Define $f(\epsilon) = \det(\mathcal{L}(x^i,\alpha^j)-\mathcal{L}(x,\alpha))_{ij}$ by rearranging the $n\times n$ matrices into column vectors of size $n^2$, and observe that this is a polynomial function of $\epsilon$. Using the multilinearity of the determinant and the fact that the derivatives of order larger than $2$ of each of the columns vanish at $0$, one gets that the derivative of order $2n^2$ at $0$ of $f$ is a non-zero multiple of $\det(2e_{ij})_{ij} = 2^{n^2}$. Hence $f$ is a non-zero polynomial, so up to choosing $\epsilon$ slightly smaller, we can suppose that $f(\epsilon)\neq 0$, which yields the result.
\end{proof}

Now, we give a necessary and sufficient condition for a subgroup of $\SLpm(n,\R)$ to preserve a sharp open convex cone. To avoid any ambiguity, we clarify first that we say a linear combination of continuous functions on a topological space $X$ is positive if it is positive as a continuous function on $X$.

\begin{proposition}\label{prop:gpPreservesDomainNSC}
   Let $\Gamma$ be a subgroup of $\SLpm(n,\R)$. Then $\Gamma$ preserves a sharp open convex cone if and only if it admits a positive generic s-proper matrix coefficient.

\end{proposition}

\begin{proof}
   Suppose first that $\Gamma$ preserves a sharp open convex cone $C\subset\R^n$. Then for any $x\in C$ and $\alpha\in C^*$, $\norm{\gamma\cdot x}\to +\infty$ when $\gamma\to\infty$ in $\Gamma$, since the action of $\Gamma$ on $\Omega$ is proper by Proposition \ref{prop:autGpActsProp}, and $\Gamma$ preserves each leaf of the foliation of $C$ by Vinberg hypersurfaces by Theorem \ref{thm:VinHypersurfaces}. Moreover, since $\alpha\in C^*$, there is $A>0$ such that $\alpha(y) < -A$ for any $y\in C$ of norm $1$. Hence, we have
   \begin{align*}
      \alpha(\gamma\cdot x) &= \norm{\gamma\cdot x}\alpha\left(\frac{\gamma\cdot x}{\norm{\gamma\cdot x}}\right) \\
      &< -A\norm{\gamma\cdot x} \\ &\tosub{\substack{\gamma\to\infty \\ \gamma \in \Gamma}} -\infty
   \end{align*}
   so $\gamma\mapsto -\alpha(\gamma\cdot x)$ is a positive s-proper matrix coefficient of $\Gamma$. Since this holds for all $(x,\alpha)\in C\times C^\ast$, the associated s-proper matrix coefficients of $\Gamma$ are generic by Lemma \ref{lm:genPGLCAreGenPLC}.

   Conversely, suppose there is a positive generic s-proper matrix coefficient of $\Gamma$. We can find open subsets $U$ and $V$ of $\R^n$ and $(\R^n)^\ast$ respectively such that
   \begin{alignat*}{3}
      \alpha(\gamma\cdot x)&\tosub{\substack{\gamma\to\infty \\ \gamma\in\Gamma}}+\infty &&\quad\text{and}\quad & \alpha(x) > 0
   \end{alignat*}
   for all $(x,\alpha)\in U\times V$.

   Hence, if $C$ is the smallest open convex cone containing the orbit of $U$, we have $-V\subset C^\ast$. Thus $C^\ast$ has non-empty interior, so $C$ is sharp by Lemma \ref{lm:sharpConvConeNSC}. Clearly $C$ is $\Gamma$-invariant.
   
\end{proof}

\begin{remark}
   We may see Proposition \ref{prop:gpPreservesDomainNSC} as a generalization of the necessary condition due to Y. Benoist for a unipotent element to preserve a properly convex domain (see \cite[\Lm 2.3]{benoist2006}).

   Indeed, a unipotent element may be conjugated to a block diagonal matrix with Jordan blocks as diagonal entries. The group generated by a single Jordan block of size $k\times k$ is
   \begin{equation*}
      \left\{\begin{pmatrix}
         1 & n & \frac{n^2}{2} & \cdots & \frac{n^{k-1}}{(k-1)!} \\
         0 & 1 & n             & \cdots & \frac{n^{k-2}}{(k-2)!} \\
         0 & 0 & 1             & \cdots & \frac{n^{k-3}}{(k-3)!} \\
         \vdots & \vdots & \vdots & \ddots & \vdots \\
         0 & 0 &           0   &  0    &   1
      \end{pmatrix}, n\in \Z\right\}
   \end{equation*}
   whose entries admit a generic s-proper matrix coefficient if and only if $n^{k-1}$ is s-proper, if and only if $k$ is odd. Therefore, a unipotent element preserves a properly convex domain if and only if the maximal size of its Jordan blocks is odd.
\end{remark}

We can also use this proposition to produce examples of weakly unipotent subgroups of $\SLpm(n,\R)$ preserving properly convex domains.

\begin{example}
   The following group is generated by a parabolic automorphism preserving a properly convex domain, but no strictly convex domains, nor any decomposable domain: 
   \begin{equation*}
      \left\{\begin{psmallmatrix}
         1 & n & \frac{n^2}2 & & & & & & \\
           & 1 & n     & & & & & & \\
           &   & 1     & & & & & & \\
           &   &       &  \cos n & \sin n &  n\cos n & n\sin n &  \frac{n^2}2\cos n & \frac{n^2}2\sin n \\
           &   &       & -\sin n & \cos n & -n\sin n & n\cos n & -\frac{n^2}2\sin n & \frac{n^2}2\cos n \\
           &   &       &         &        &   \cos n &  \sin n &    n \cos n  &   n \sin n \\
           &   &       &         &        &  -\sin n &  \cos n &   -n \sin n  &   n \cos n \\
           &   &       &         &        &          &         &      \cos n  &     \sin n \\
           &   &       &         &        &          &         &     -\sin n  &     \cos n
      \end{psmallmatrix},n\in\Z\right\}
   \end{equation*}

   Since $2 + \cos n \geq 1$ for all $n$, the matrix coefficient given by 
   \begin{alignat*}{3}
      \alpha &= \begin{psmallmatrix}1 & 0 & 0 & 1 & 1 & 0 & 0 & 0 & 0\end{psmallmatrix}&\quad\text{and}\quad &&x &= \begin{psmallmatrix}0 \\ 0 \\ 2 \\ 0 \\ 0 \\ 0 \\ 0 \\ 1/2 \\ 1/2\end{psmallmatrix}
   \end{alignat*}
   is s-proper bounded below as the coefficient in front of the $n^2/2$ term is $2 + \cos n\geq 1$, so that this term dominates all other terms for $n$ large. It is clear that perturbing slightly both $\alpha$ and $x$ would yield a coefficient that is always larger or equal to $\frac 1 2$, for instance, giving the genericity by Lemma \ref{lm:genPGLCAreGenPLC}.

   Moreover, the group does not preserve any strictly convex domain by Theorem \ref{thm:characHoloStrictlyConvCusps}. It does not preserve any decomposable domain either since it has only 2 indecomposable real blocks, so a decomposable domain would be the join of two domains preserved by each of the two blocks. However, the second block (with complex eigenvalues) does not preserve any properly convex domain as there is no positive linear combination of its entries.
\end{example}

\section{\texorpdfstring{$P_1$}{P\_1}-divergence and properly convex domains}

In this section, we recall well-known properties of divergent subgroups of $\SLpm(n,\allowbreak\R)$.
Recall the Cartan decomposition of $\SLpm(n,\allowbreak\R)$:
\begin{equation*}
   \SLpm(n,\R) = K A K = KA^+K
\end{equation*}
where $K = O(n)$ is a maximal compact subgroup of $\SLpm(n,\R)$, $A<\SLpm(n,\R)$ is the diagonal subgroup, and $A^+<A$ is the subgroup consisting of diagonal matrices $\diag(\sigma_1,\dots,\sigma_n)$ with $\sigma_1\geq \dots \geq\sigma_n>0$.

The Cartan projection of an element $g\in\SLpm(n,\R)$ is the ordered $n$-tuple of its \emph{singular values} $\sigma_1(g)\geq \dots\geq \sigma_n(g) >0$ such that $g = K\diag(\sigma_1(g),\dots,\sigma_n(g)) L$ for some $K,L\in\O(n)$. Observe that the singular values of $g$ are the eigenvalues of the symmetric positive definite matrix $\sqrt{g\transpose{g}}$, so they are well-defined.

The $i$-th singular value $\sigma_i(g)$ of $g$ corresponds to the length of the $i$-th largest axis of the image of the unit ball $\Ell$ under $\sqrt{g\transpose{g}}$.

\begin{definition}
   An unbounded sequence $(\gamma_m)_m$ of $\SLpm(n,\R)$ is $P_k$-divergent if $\sigma_k(\gamma_m)/\sigma_{k+1}(\gamma_m)\to +\infty$.

   A group $\Gamma<\SLpm(n,\R)$ is $P_k$-divergent if all unbounded sequence of $\Gamma$ are $P_k$-divergent.
\end{definition}

Since the ratios $\sigma_k(g)/\sigma_{k+1}(g)$ are well-defined for any $g\in \PGL(n,\R)$, the same definition makes sense for subgroups of $\PGL(n,\R)$. All the discussion below directly generalizes to $\PGL(n,\R)$ by this observation.

\begin{remark}\label{rmk:divergentSequenceDynamics}
   Let $(\gamma_m)_m$ be a $P_k$-divergent sequence of $\SLpm(n,\R)$. By the Cartan decomposition, we can write
   \begin{equation*}
      \gamma_m = K_m\diag(\sigma_1(\gamma_m),\dots,\sigma_n(\gamma_m)) L_m
   \end{equation*}
   where $K_m,L_m\in O(n)$. We may assume up to extracting subsequences that $K_m\to K$ and $L_m\to L$. Moreover, if $l\leq k$ is the first index such that the sequence $\sigma_l(\gamma_m)/\sigma_{l+1}(\gamma_m)$ is unbounded, we may suppose up to further extracting that $\gamma_m$ is $P_l$-divergent, so that $[\gamma_m]\to[K\Delta L]$ in $\P(\End(\R^n))$, where $\Delta = \diag(1,s_2,\dots,s_l,\allowbreak 0,\dots,0)$ for some $s_i$ such that $1\geq s_2\geq \dots\geq s_l >0$.

   It follows that for any $x\in \P(\R^n)-\P(L^{-1}\ker\Delta)$, $\gamma_m x\to [K\Delta L x]$. In particular, if $U$ is an open set of $\P(\R^n)$ whose closure does not meet $\P(L^{-1}\ker\Delta)$, then $\gamma_m U\to \P(K\Delta L U)$ for the Hausdorff metric, and $l-1$ is the dimension of the Hausdorff limit of $\gamma_m U$.
\end{remark}

It is not difficult to see that a $P_k$-divergent subgroup of $\SLpm(n,\R)$ is also $P_{n-k}$-divergent.

\begin{lemma}\label{lm:dualityOfDivergentGroups}
   Let $(\gamma_m)_m$ be a $P_k$-divergent sequence of $\SLpm(n,\R)$. Then $(\gamma_m^{-1})_m$ is $P_{n-k}$-divergent.

   In particular, any $P_k$-divergent subgroup of $\SLpm(n,\R)$ is also $P_{n-k}$-divergent.
\end{lemma}
\begin{proof}
   Observe that if $g\in\SLpm(n,\R)$, we can write $g = K\diag(\sigma_1(g),\dots,\sigma_n(g))L$ with $K,L\in\O(n)$ and $\sigma_1(g)\geq \dots \geq\sigma_n(g) >0$. It follows that
   \begin{align*}
      g^{-1} &= L^{-1}\diag(\sigma_1(g)^{-1},\dots,\sigma_n(g)^{-1}) K^{-1} \\
      &= K' \diag(\sigma_n(g)^{-1},\dots,\sigma_1(g)^{-1}) L'
   \end{align*}
   where $K',L'\in\O(n)$ and $\sigma_n(g)^{-1}\geq \dots\geq\sigma_1(g)^{-1}>0$. Hence $\sigma_i(g^{-1}) = \sigma_{n-i+1}(g)^{-1}$.

   It follows that if $(\gamma_m)_m$ is $P_k$-divergent, then we have
   \begin{equation*}
      \sigma_{n-k}(\gamma_m^{-1})/\sigma_{n-k+1}(\gamma_m^{-1}) = \sigma_k(\gamma_m)/\sigma_{k+1}(\gamma_m) \to +\infty
   \end{equation*}
   so $(\gamma_m^{-1})_m$ is $P_{n-k}$-divergent.
\end{proof}

\subsection{Limit sets of divergent groups}

Let $\Gamma$ be a $P_k$-divergent subgroup of $\SLpm(n,\R)$. By considering the embedding given by the composition $\SLpm(n,\R)\to\PGL(n,\R)\to\P(\End(\R^n))$, we can associate to unbounded sequences in $\Gamma$ an attractive (\resp repulsive) limit flag in $\grass_k(\R^n) = \flag_k(\R^n)$ (\resp in $\grass_{n-k}(\R^n) = \flag_{n-k}(\R^n)$).

Let $(\gamma_m)_m$ be a $P_k$-divergent sequence of $\SLpm(n,\R)$. We can write 
\begin{equation*}
   \gamma_m = K_m\diag(\sigma_1(\gamma_m),\dots,\sigma_n(\gamma_m)) L_m
\end{equation*}
where $K_m,L_m\in\O(n)$. Up to extracting a subsequence, we can suppose that $K_m\to K$ and $L_m\to L$ in $\O(n)$. The unstable $k$-dimensional (\resp stable $n-k$-dimensional) subspace of $\gamma_{\bigcdot}:=(\gamma_m)_m$ is $E_k^+(\gamma_{\bigcdot}) := K\gen{e_1,\dots,e_k}$ (\resp $E_{n-k}^-(\gamma_{\bigcdot}) :=L\gen{e_{k+1},\dots,e_n}$). Observe that if $\gamma_{\bigcdot}$ is also $P_{n-k}$-divergent, then $E_{n-k}^-(\gamma_{\bigcdot})=E_k^+(\gamma_{\bigcdot}^{-1}) $, which is well-defined since $\gamma^{-1}_{\bigcdot}$ is also $P_{n-k}$-divergent by Lemma \ref{lm:dualityOfDivergentGroups}.

For $k=1$, we see that the unstable subspaces behave like attracting points for the dynamics on the projective space.

\begin{lemma}\label{lm:P1UnstableSpaceIsAttractive}
   Let $(\gamma_m)_m$ be a $P_1$-divergent sequence of $\SLpm(n,\R)$ whose stable and unstable subspaces are defined. Then, for any point $x\in\P(\R^n)-\P(E_{n-1}^-(\gamma_{\bigcdot}))$, the sequence $\gamma_m x$ converges to the point $\P(E_1^+(\gamma_{\bigcdot}))$ of $\P(\R^n)$, uniformly on compact sets.
\end{lemma}
\begin{proof}
   Write $\gamma_m = K_m\diag(\sigma_1(\gamma_m),\dots,\sigma_n(\gamma_m)) L_m$ with $K_m\to K$ and $L_m \to L$ are sequences in $\O(n)$, and $\sigma_1(\gamma_m)/\sigma_2(\gamma_m)\to +\infty$. Observe that $\gamma_m \to [K\Delta L]$ in $\P(\End(\R^n))$, where $\Delta = \diag(1,0,\dots,0)$. It follows that for any $x\in\P(\R^n)-\P(L^{-1}\ker\Delta) = \P(\R^n)-\P(E_{n-1}^-(\gamma_{\bigcdot}))$, $\gamma_m x \to \P(K\im\Delta) = \P(E_1^+(\gamma_{\bigcdot}))$.
\end{proof}

A reciprocal to that statement also holds.

\begin{lemma}
   Let $(\gamma_m)_m$ be an unbounded sequence of $\SLpm(n,\R)$. Suppose that there is a point $x\in\P(\R^n)$ and a hyperplane $H\in\P((\R^n)^\ast)$ such that $\gamma_m y\to x$ for all $y\not\in H$, uniformly on compact sets. Then $(\gamma_m)_m$ is $P_1$-divergent.
\end{lemma}
\begin{proof}
   Write $\gamma_m = K_m\diag(\sigma_1(\gamma_m),\dots,\sigma_n(\gamma_m))L_m$ with $K_m,L_m\in \O(n)$. Up to extracting a subsequence, we may assume that $K_m\to K\in \O(n)$ and $L_m\to L\in\O(n)$. Up to further extracting, we may assume that there is $k$ minimal such that $\sigma_1(\gamma_m)/\sigma_{k+1}(\gamma_m)\to +\infty$. It follows that for any open subset $U$ of $\P(\R^n)-H$, $\gamma_m U$ converges for the Hausdorff metric to a $k-1$-dimensional subset of $\P(\R^n)$. By hypothesis however, we know that $\gamma_m U\to\{x\}$ for the Hausdorff metric, hence $k=1$ and $(\gamma_m)_m$ is $P_1$-divergent.

   Applying the previous argument to an arbitrary subsequence of $(\gamma_m)_m$ shows that all subsequences of $(\gamma_m)_m$ admit a $P_1$-divergent subsequence. Hence $(\gamma_m)_m$ itself is $P_1$-divergent.
\end{proof}

Similar results hold for $P_1$-divergent sequences acting on $\flag_{1,n-1}(\R^n)$. This justifies Definition \ref{def:divergentGroup} in the introduction. We take what follows Definition \ref{def:divergentGroup} as our definition of limit sets here.

\subsection{Divergent groups preserving properly convex domains}

\begin{lemma}\label{lm:(un)stableSpaceSupportsDomain}
   Let $\Omega$ be a properly convex domain, and $(\gamma_m)_m$ be a $P_k$-divergent sequence of $\Aut(\Omega)$ whose unstable and stable subspaces are defined. Then both $\P(E_k^+(\gamma_{\bigcdot}))$ and $\P(E_{n-k}^-(\gamma_{\bigcdot}))$ are supporting subspaces to $\Omega$.
\end{lemma}
\begin{proof}
   Let $U$ be an open subset of $\Omega-\P(E_{n-k}^-(\gamma_{\bigcdot}))$, then $\gamma_m U$ converges to an open subset $X$ of $\P(E_k^+(\gamma_{\bigcdot}))$ for the Hausdorff metric by Remark \ref{rmk:divergentSequenceDynamics}. Moreover, since $\gamma_m\in\Aut(\Omega)$ for all $m$ and $\gamma_m$ is unbounded, Proposition \ref{prop:autGpActsProp} implies that $X$ lies in $\partial\Omega$. It follows that $\P(E_k^+(\gamma_{\bigcdot}))$ is a supporting subspace to $\Omega$.

   The result for $\P(E_{n-k}^-(\gamma_{\bigcdot}))$ follows from Lemma \ref{lm:dualityOfDivergentGroups} and the observation that $E_{n-k}^-(\gamma_{\bigcdot}) = E_{n-k}^+(\gamma^{-1}_{\bigcdot})$.
\end{proof}

It is a classical fact that discrete subgroups of automorphisms of a strictly convex domain are $P_1$-divergent.

\begin{proposition}\label{prop:autOfStrictlyConvDomainsAreP1Div}
   Let $\Omega$ be a strictly convex domain. Then $\Aut(\Omega)$ is $P_1$-divergent.
\end{proposition}
\begin{proof}
   Let $\gamma_m\in\Aut(\Omega)$ be an unbounded sequence. We can write $\gamma_m = K_m \Delta_m L_m$ where $K_m,L_m\in \O(n)$ and $\Delta_m = \diag(\sigma_1(\gamma_m),\dots,\sigma_n(\gamma_m))$. Since $\O(n)$ is compact, we can suppose up to extracting a subsequence $(\gamma_{\psi(m)})_m$ that $K_{\psi(m)}\to K$ and $L_{\psi(m)}\to L$. 
   
   Let $k$ be the smallest integer such that $\sigma_k(\gamma_{\psi(m)})/\sigma_{k+1}(\gamma_{\psi(m)}) \to\infty$. Up to passing to another subsequence $(\gamma_{\phi(m)})_m$, we can suppose that $[\Delta_{\phi(m)}]$ converges to $[D]:=[\diag(a_1,\dots,a_k,0\dots,0)]$ in $\P(\End(\R^n))$ (where $a_1\geq \dots\geq a_k >0$), so that $\gamma_{\phi(m)}\to [T]:= [kDl]$ in the same space. It follows that for any $x\in \P(\R^n)-\P(\ker T)$, $\gamma_{\phi(m)} x\to Tx$.

   In particular, if $U$ is an open subset of $\Omega-\P(\ker T)$, $\gamma_{\phi(m)} U\to TU$ which is relatively open in the $k-1$-dimensional subspace $\P(\im T)$ of $\R^n$. Since the action of $\Aut(\Omega)$ on $\Omega$ is proper by Proposition \ref{prop:autGpActsProp}, all limit points of $\gamma_{\phi(m)} U$ lie in $\partial\Omega$. Therefore, $[T]U$ is a subset of a face of $\partial\Omega$ of dimension at least $k-1$. Since $\Omega$ is strictly convex, we have $k = 1$, so $(\gamma_{\phi(m)})_m$ is $P_1$-divergent.

   Hence, we can extract from any subsequence of $(\gamma_m)_m$ a further subsequence along which $\sigma_1/\sigma_2$ diverges to $\+\infty$. It follows that $\sigma_1(\gamma_m)/\sigma_2(\gamma_m)\to +\infty$. That is, $\Aut(\Omega)$ is $P_1$-divergent.
\end{proof}

\begin{proposition}
   If $\Gamma<\Aut(\Omega)$ is $P_1$-divergent, then its $P_1$-limit set $\Lambda_1(\Gamma)$ equals its orbital limit set $\Lambda_\Omega(\Gamma)$.
\end{proposition}
\begin{proof}
   By Lemma \ref{lm:(un)stableSpaceSupportsDomain}, given any $P_1$-divergent sequence of $\Gamma$ and any point $x\in\Omega$, $x$ is transverse to $\P(E_{n-1}^-(\gamma_{\bigcdot}))$, so $\gamma_m x\to\P(E_1^+(\gamma_{\bigcdot}))\in\partial\Omega$ by Lemma \ref{lm:P1UnstableSpaceIsAttractive}. In particular, $\Lambda_1(\Gamma)\subset\Lambda_\Omega(\Gamma)$.

   For the converse, consider a point $z\in\Lambda_\Omega(\Gamma)$, a sequence $(\gamma_m)_m$ in $\Gamma$ and a point $x\in\Omega$ such that $\gamma_m x\to z$. Since $(\gamma_m)_m$ is unbounded, it is $P_1$-divergent, so up to extraction, we can suppose that the stable and unstable subspaces of $(\gamma_m)_m$ are defined, so that $\gamma_m x'\to \P(E_1^+(\gamma_{\bigcdot})$ for any $x'\in\Omega$ by Lemma \ref{lm:(un)stableSpaceSupportsDomain}. In particular, $\gamma_m x\to\P(E_1^+(\gamma_{\bigcdot}))$, so we have $z = \P(E_1^+(\gamma_{\bigcdot}))\in\Lambda_1(\Gamma)$.
\end{proof}

\section{Characterization of strictly convex and round cusp holonomies}

\subsection{The conditions}

Our goal is to give a characterization of the holonomies of strictly convex and round cusps. For this purpose, we introduce the following conditions on an unbounded subgroup $\Gamma$ of $\SLpm(n,\R)$:
\begin{conditions}
   \conditem{(WU)}{cond:gpIsWeaklyUnip} $\Gamma$ is weakly unipotent;
   \conditem{(GP+)}{cond:existsPosGenPGLC} there is a positive generic s-proper matrix coefficient of $\Gamma$;
   \conditem{(GP)}{cond:existsGenPLC} there is a generic s-proper linear combination of the entries of $\Gamma$;
   \conditem{(Tr)}{cond:topRowDominates} there is a pair $(p,\phi)\in (\R^n-\{0\})\times((\R^n)^\ast-\{0\})$ of $\Gamma$-fixed vectors such that $\phi(p) = 0$ and in any basis of $\R^n$ whose first element is $p$, all dominating entries of $\Gamma$ are on the top row;
   \conditem{(TRe)}{cond:topRightDominates} there is a pair $(p,\phi)\in (\R^n-\{0\})\times((\R^n)^\ast-\{0\})$ of $\Gamma$-fixed vectors such that $\phi(p) = 0$ and in any basis of $\R^n$ whose first element is $p$ and the $n-1$ first elements span $\ker\phi$, the top right entry of $\Gamma$ is the only dominating entry and has constant sign outside of a bounded subset of $\Gamma$.
\end{conditions}

\begin{remark}\label{rmk:TrTReMatrixForm}
   The existence of a pair $(p,\phi)$ as mandated in conditions \ref{cond:topRowDominates} and \ref{cond:topRightDominates} is equivalent to saying that there is a basis of $\R^n$ in which all elements of $\Gamma$ are of the form
   \begin{equation*}
      \begin{pmatrix}
         1 & * & * \\
         0 & * & * \\
         0 & 0 & 1
      \end{pmatrix}
   \end{equation*}

   However, observe that condition \ref{cond:topRowDominates} also applies when $\Gamma$ is only in a basis such that its elements are of the form
   \begin{equation*}
      \begin{pmatrix}
         1 & * & * \\
         0 & * & * \\
         0 & * & *
      \end{pmatrix}
   \end{equation*}
\end{remark}
\begin{remark}\label{rmk:TReImpliesTReProper}
   Condition \ref{cond:topRightDominates} implies that the top right entry is also s-proper. Indeed, if $\gamma\to\infty$ in $\Gamma$, $\norm{\gamma}\to\infty$, but since $\gamma\mapsto\gamma_{1n}$ dominates all other entries of $\Gamma$, $\norm{\gamma}\sim \abs{\gamma_{1n}}$, so that $\gamma\mapsto \gamma_{1n}$ is s-proper since it does not change sign outside of a bounded subset of $\Gamma$.
\end{remark}

\begin{remark}\label{rmk:TrImpliesGenPLCAreGenPGLC}
   If $\Gamma<\SLpm(n,\R)$ satisfies \ref{cond:topRowDominates}, then there is a generic s-proper linear combination of the entries of $\Gamma$ if and only if there is a generic s-proper matrix coefficient of $\Gamma$. This is because a generic s-proper linear combination of the entries of $\Gamma$ will have the order of the entries of $\Gamma$ by Proposition \ref{prop:genPLCOrder}, so the restriction of this linear combination to the entries of the top row of $\Gamma$ is also s-proper, because the restriction to the other entries of $\Gamma$ is a little $o$ of the whole linear combination. Since the first linear combination is generic, so is its restriction. But this is clearly a matrix coefficient of $\Gamma$ by choosing $\alpha=\begin{psmallmatrix}1 & 0 &\cdots & 0\end{psmallmatrix}$, and the coordinates of $x$ to be the coefficients of the restriction to the entries on the top row.
\end{remark}

\begin{remark}\label{rmk:TReImpliesGP+}
   It is not hard to see that \ref{cond:topRightDominates} implies \ref{cond:existsPosGenPGLC}. Indeed, in a basis as specified by \ref{cond:topRightDominates}, the top left entry will be $1$, and the top right entry is s-proper by Remark \ref{rmk:TReImpliesTReProper}. 
   Then, if $A$ is large enough and $\epsilon\in\{\pm 1\}$ has the right sign, the matrix coefficient given by $\alpha=\begin{psmallmatrix}1&0&\cdots&0\end{psmallmatrix}$ and $x=(A,0,\dots,0,\epsilon)$ is a positive generic s-proper matrix coefficient of $\Gamma$.

      Similarly, if $\Gamma$ satisfies \ref{cond:existsGenPLC} and \ref{cond:topRowDominates}, it also satisfies \ref{cond:existsPosGenPGLC}, we can find a generic s-proper matrix coefficients of $\Gamma$ given by $\alpha=\begin{psmallmatrix}1&0&\cdots&0\end{psmallmatrix}$ and some $x\in\R^n$ by Remark \ref{rmk:TrImpliesGenPLCAreGenPGLC}, so up to changing the sign of $x$ and increasing its first coefficient, we get a positive generic s-proper matrix coefficient of $\Gamma$.
\end{remark}

\begin{theorem}\label{thm:characHoloStrictlyConvCusps}
   Let $\Gamma$ be an unbounded subgroup of $\PGL(n,\R)$. The following are equivalent:
   \begin{enumerate}[(i)]
      \item there is a $\Gamma$-invariant strictly convex domain $\Omega\subset\P(\R^n)$ on which the action of $\Gamma$ is doubly elementary of the parabolic type fixing a point $\xi\in\partial\Omega$ and a hyperplane $H$ supporting $\Omega$ at $\xi$, such that $\partial\Omega-\{\xi\}$ is the graph of a strictly convex analytic function in the affine chart $\P(\R^n)-H$;
      \item there is a $\Gamma$-invariant strictly convex domain $\Omega\subset\P(\R^n)$ on which the action of $\Gamma$ is elementary of the parabolic type;
         \label{item:gpGivesStrictlyConvCusp}
      \item $\Gamma$ is $P_1$-divergent with $\Lambda_1(\Gamma)$ a point, and it preserves a properly convex domain;
      \item a lift of $\Gamma$ to $\SLpm(n,\R)$ satisfies \ref{cond:gpIsWeaklyUnip}, \ref{cond:existsPosGenPGLC} and \ref{cond:topRowDominates};
         \label{item:strongCondStrictlyConvCusp}
      \item a lift of $\Gamma$ to $\SLpm(n,\R)$ satisfies \ref{cond:existsGenPLC} and \ref{cond:topRowDominates}.
         \label{item:weakCondStrictlyConvCusp}
   \end{enumerate}

   In particular, if $\Gamma$ is discrete, then $\Gamma$ is the holonomy of a strictly convex cusp if and only if it satisfies \ref{cond:existsGenPLC} and \ref{cond:topRowDominates}. In this case, $\Gamma$ is virtually nilpotent.
\end{theorem}

\begin{theorem}\label{thm:characHoloRoundCusps}
   Let $\Gamma$ be an unbounded subgroup of $\PGL(n,\R)$. The following are equivalent:
   \begin{enumerate}[(i)]
      \item there is a $\Gamma$-invariant round domain $\Omega\subset\P(\R^n)$ on which the action of $\Gamma$ is doubly elementary of the parabolic type fixing a point $\xi\in\partial\Omega$ and a hyperplane $H$ supporting $\Omega$ at $\xi$, such that $\partial\Omega-\{\xi\}$ is the graph of a strictly convex analytic function in the affine chart $\P(\R^n)-H$;
         \label{item:gpGivesRoundAnalyticCusp}
      \item there is a $\Gamma$-invariant round domain $\Omega\subset\P(\R^n)$ on which the action of $\Gamma$ is elementary of the parabolic type;
         \label{item:gpGivesRoundCusp}
      \item $\Gamma$ is $P_1$-divergent with $\Lambda_{1,n-1}(\Gamma)$ a point, and it preserves a properly convex domain;
         \label{item:P1DivGpGivesRoundCusp}
      \item a lift of $\Gamma$ to $\SLpm(n,\R)$ satisfies \ref{cond:gpIsWeaklyUnip}, \ref{cond:existsPosGenPGLC} and \ref{cond:topRightDominates};
         \label{item:strongCondRoundCusp}
      \item a lift of $\Gamma$ to $\SLpm(n,\R)$ satisfies \ref{cond:topRightDominates}.
         \label{item:weakCondRoundCusp}
   \end{enumerate}

   In particular, if $\Gamma$ is discrete, it is the holonomy of a round cusp if and only if it satisfies \ref{cond:topRightDominates}. In this case, $\Gamma$ is virtually nilpotent.
\end{theorem}

\begin{proof}[Proof of Theorems \ref{thm:characHoloStrictlyConvCusps} and \ref{thm:characHoloRoundCusps}]
   Proposition \ref{prop:necessaryCond} proves that \ref{item:P1DivGpGivesRoundCusp} implies \ref{item:strongCondRoundCusp} in both the round and strictly convex case, and it is clear that \ref{item:strongCondRoundCusp} implies \ref{item:weakCondRoundCusp}. Propositions \ref{prop:suffCondStrictlyConvCase} and \ref{prop:suffCondRoundCase} show that \ref{item:weakCondRoundCusp} implies \ref{item:gpGivesRoundAnalyticCusp} in both cases, and it is clear that \ref{item:gpGivesRoundAnalyticCusp} implies \ref{item:gpGivesRoundCusp}.

   For \ref{item:gpGivesRoundCusp} implies \ref{item:P1DivGpGivesRoundCusp}, Proposition \ref{prop:autOfStrictlyConvDomainsAreP1Div} shows that $\Gamma$ is $P_1$-divergent. Since $\Gamma$ is elementary of the parabolic type, it is also doubly elementary and weakly unipotent by Proposition \ref{prop:CLTParabElemSgpSummary}, so it preserves algebraic horospheres centered at a pair $(\xi,H)$, where $\xi\in\partial\Omega$ and $H$ is a supporting hyperplane to $\Omega$ at $\xi$. Since $\Omega$ is strictly convex, these algebraic horospheres meet $\partial\Omega$ exactly at $\xi$, so we must have $\gamma x\to \xi$ when $\gamma\to\infty$ in $\Gamma$, for any $x\in\Omega$. By Lemma \ref{lm:(un)stableSpaceSupportsDomain}, this implies that $\Lambda_1(\Gamma)=\{\xi\}$. In the round case, a duality argument shows that $\Lambda_{n-1}(\Gamma) = \{H\}$ is a point as well, so we must have that $\Lambda_{1,n-1}(\Gamma) = \{(p,H)\}$ is a point.

   In the case where $\Gamma$ is discrete, it is virtually nilpotent by \ref{cond:gpIsWeaklyUnip} and \cite[\Prop 4.11]{cooperLongTillmann2015}.
\end{proof}

\subsection{The conditions are necessary}

\begin{proposition}\label{prop:necessaryCond}
   Let $\Gamma$ be a $P_1$-divergent subgroup of $\PGL(n,\R)$ preserving a properly convex domain $\Omega$, and suppose that $\Lambda_1(\Gamma)$ has cardinal $1$. Then there is a lift of $\Gamma$ to $\SLpm(n,\R)$ that satisfies \ref{cond:gpIsWeaklyUnip}, \ref{cond:existsPosGenPGLC} and \ref{cond:topRowDominates}.

   If moreover $\Lambda_{1,n-1}(\Gamma)$ has cardinal $1$, then we can find a lift of $\Gamma$ to $\SLpm(n,\R)$ that also satisfies \ref{cond:topRightDominates}.
\end{proposition}
\begin{proof}
   Let $\Omega$ be a $\Gamma$-invariant properly convex domain, and $C$ be a lift of $\Omega$. The lift of $\Gamma$ to $\SLpm(n,\R)$ preserving $C$ satisfies \ref{cond:existsPosGenPGLC} by Proposition \ref{prop:gpPreservesDomainNSC}. Observe that $\Gamma$ contains no loxodromic elements, otherwise $\Lambda_1(\Gamma)$ would contain at least its attractive and repulsive points in $\P(\R^n)$. It follows that $\Gamma$ satisfies \ref{cond:gpIsWeaklyUnip} by Lemma \ref{lm:transLengthIsRatioEV}.

   To prove the last statement about the dominating entries of $\Gamma$, let $\xi$ denote the unique $P_1$-limit point of $\Gamma$ in $\P(\R^n)$. By Lemma \ref{lm:(un)stableSpaceSupportsDomain}, we have $\xi\in\partial\Omega$, and since $\Gamma$ has no loxodromic element and is unbounded, \cite[\Cor 4.7]{cooperLongTillmann2015} shows that there is a $\Gamma$-invariant supporting hyperplane $H$ to $\Omega$ at $\xi$. Let $(p,\phi)\in \overline{C}\times \overline{C^\ast}$ be lifts of $\xi$ and $H$: they are fixed by $\Gamma$ since it is weakly unipotent and preserves $C$. We will show that the pair $(p,\phi)$ is as required in \ref{cond:topRowDominates}. We place ourselves in a basis of $\R^n$ with $e_1 = p$ and $e_n^\ast = \phi$.

   Observe that by Lemma \ref{lm:(un)stableSpaceSupportsDomain}, we have $\gamma_n x\to \P(E_1^+(\gamma_{\bigcdot}))$ for any unbounded sequence admitting stable and unstable subspaces in $\Gamma$ and $x\in\Omega$. Since moreover $\Lambda_1(\Gamma)=\{\xi\}$, one can extract out of any unbounded sequence of $\Gamma$ a subsequence converging to $\xi$ in $\Omega$. It follows that $\gamma x\to\xi$ when $\gamma\to\infty$ in $\Gamma$ for any $x\in \Omega$.

   Pick $\tilde{x}\in C$ a lift of $x$ to $\R^n$. Since $\gamma x\to\xi$ when $\gamma\to\infty$ in $\Gamma$, we have that $\gamma\mapsto (\gamma\cdot\tilde{x})_1$ dominates $\gamma\mapsto (\gamma\cdot\tilde{x})_i$ for $i\geq 2$.

   Recall from the proof of Proposition \ref{prop:gpPreservesDomainNSC} and Lemma \ref{lm:genPGLCAreGenPLC} that for any $\alpha\in C^\ast$, $\gamma\mapsto\alpha(\gamma\cdot\tilde{x})$ is a generic s-proper linear combination of the entries of $\Gamma$. Observe that it is also a generic s-proper linear combination of the entries of $\gamma\mapsto\gamma\cdot\tilde{x}$. Therefore, Proposition \ref{prop:genPLCOrder} implies that the entries of $\Gamma$ have the same order as the entries of $\gamma\mapsto\gamma\cdot \tilde{x}$, which have the same order as $\gamma\mapsto(\gamma\cdot\tilde{x})_1$ by the last paragraph.

   Define $\tilde{y} = \tilde{x} + \epsilon e_j$ where $\epsilon$ is small enough that $\tilde{y}$ projects into $\Omega$. The $(i,j)$-th entry of $\Gamma$ is exactly $\gamma\mapsto (\gamma\cdot\frac1\epsilon (y-x))_i$, which is a linear combination of entries dominated by the entries of $\Gamma$ whenever $i\geq 2$. Hence it cannot be dominating, and the dominating entries of $\Gamma$ are all on the top row, \ie $\Gamma$ satisfies \ref{cond:topRowDominates}.

   \medskip

   In order to get item \ref{cond:topRightDominates} in the round case, observe that we can apply what is above to both the action of $\Gamma$ on $\Omega$ and $\Omega^\ast$. Since the dual action of $\Gamma$ is by its transpose, we need to reorder the dual basis as $(e_n^\ast,\dots,e_1^\ast)$ to get a matrix form as desired. This transformation takes the top row of $\Gamma$ to the rightmost column, leaving the top right entry in place. Hence, the dominating entries of $\Gamma$ must be at the same time on its top row and its rightmost column, so the only dominating entry is in the top right corner.
\end{proof}

\subsection{The conditions are sufficient}

We now prove that the conditions given in Proposition \ref{prop:necessaryCond} are sufficient. The strategy is to show that groups satisfying these conditions preserve domains that are regular enough to apply the smoothing Lemma \ref{lm:smoothing}.

We start with the strictly convex case which is easier.

\begin{proposition}\label{prop:suffCondStrictlyConvCase}
   Let $\Gamma$ be an unbounded subgroup of $\SLpm(n,\R)$ satisfying \ref{cond:existsGenPLC} and \ref{cond:topRowDominates}. There exists a $\Gamma$-invariant strictly convex domain $\Omega$ on which the action of $\Gamma$ is doubly elementary of the parabolic type fixing $\xi\in\partial\Omega$ and a hyperplane $H$ supporting $\Omega$ at $\xi$, such that $\partial\Omega-\{\xi\}$ is the graph of a strictly convex analytic function in the affine chart $\P(\R^n)-H$.

   In particular, if $\Gamma$ is discrete, then $\Omega/\Gamma$ is a strictly convex cusp.
\end{proposition}

\begin{proof}
   We keep the notations given in condition \ref{cond:topRowDominates}.
   By Remark \ref{rmk:TReImpliesGP+}, $\Gamma$ also satisfies \ref{cond:existsPosGenPGLC}, so Proposition \ref{prop:gpPreservesDomainNSC} gives the existence of $\Gamma$-invariant domains of the form
   \begin{align*}
      \Omega_1 &= \Int \Conv (\Gamma\cdot\P(U_1)) \\
      \Omega_0 &= \Int \Conv(\Gamma\cdot\P(U))
   \end{align*}
   where $U\subset \R^n$ and $V\subset (\R^n)^\ast$ are open subsets such that for any $(x,\alpha)\in U\times V$, $\gamma\mapsto\alpha(\gamma\cdot x)$ is a negative generic s-proper matrix coefficient of $\Gamma$, and $U_1$ is an open subset of $U$ such that $\overline{U_1}\subset U$. We further assume up to choosing $U$ smaller that $\overline{\P(U)}$ does not meet $H:=\P(\ker\phi)$, so that $\Omega_0\cap H=\varnothing$.

   We now choose a basis of $\R^n$ as prescribed in condition \ref{cond:topRowDominates}. Pick any pair $(x,\alpha)\in U\times V$. Since $\gamma\mapsto\alpha(\gamma\cdot x)$ is generic by Lemma \ref{lm:genPGLCAreGenPLC}, it is also a generic s-proper linear combination of the entries of $\gamma\mapsto\gamma\cdot x$. Hence the entries of $\gamma\mapsto\gamma\cdot x$ have the same order as the entries of $\Gamma$ by Proposition \ref{prop:genPLCOrder}. But for all $i\geq 2$, $\gamma\mapsto(\gamma\cdot x)_i$ is a linear combination of dominated entries of $\Gamma$ by condition \ref{cond:topRowDominates}, so the only dominating entry of $\gamma\mapsto\gamma\cdot x$ is its first entry.

   It follows that $\gamma\mapsto(\gamma\cdot x)_1$ dominates $\gamma\mapsto(\gamma\cdot x)_i$ for all $i\geq 2$. Hence $(\gamma\cdot x)_1\sim\alpha(\gamma\cdot x)/\alpha_1$ when $\gamma\to\infty$ in $\Gamma$, and hence $\gamma\mapsto (\gamma\cdot x)_1$ is s-proper. Therefore, for all $x\in U$, $\gamma [x]\to \xi:=[p]$ when $\gamma\to\infty$ in $\Gamma$. In particular, for $x\in U_1$, this implies $\xi\in\overline{\Omega_1}$. Since $\Omega_0\cap H=\varnothing$, $\xi\in\partial\Omega_0\cap\partial\Omega_1$.

   Since $\overline{\P(U_1)}$ is a compact subset of $\Omega_0$, the closure of $\Gamma\cdot\overline{\P(U_1)}$ meets $\partial\Omega_0$ only at $\xi$ by Proposition \ref{prop:autGpActsProp} and the last paragraph. Therefore, $\overline{\Omega_1}-\{\xi\}\subset \Omega_0$, so that $\overline{\Omega_1}\cap H = \{\xi\}$.

   We now apply Lemma \ref{lm:smoothing} to $\Omega_1$ and $H$, yielding a strictly convex domain $\Omega\subset\Omega_1$ such that $\overline{\Omega}\cap H = \overline{\Omega_1}\cap H = \{\xi\}$, and $\partial\Omega - \{\xi\}$ is the graph of a strictly convex analytic function in the affine chart $\P(\R^n)-H$.

   Since the weights of $\Gamma$ associated to $p$ and $\phi$ are trivial by \ref{cond:topRowDominates}, $\Gamma$ preserves the algebraic horospheres of $\Omega$ centered at $(\xi,H)$. Since $\Omega$ is strictly convex, any hyperbolic element in $\Gamma$ would have a unique axis in $\Omega$, so it cannot preserve any algebraic horosphere. Hence, $\Gamma$ is doubly elementary of the parabolic type.
\end{proof}

\begin{remark}\label{rmk:TrImpliesStrictlyConvDomInInvariantDomain}
   Keeping the same notations as in the proof above, observe that if $\Gamma$ is already elementary of the parabolic type in a properly convex domain $\Omega'$, then there is a $\Gamma$-invariant strictly convex domain $\Omega\subset\Omega'$ such that $\overline{\Omega}\cap\partial\Omega'=\{\xi'\}$, where $\xi'$ is the fixed point of $\Gamma$ in $\partial\Omega'$.

   Indeed, we can first observe that $\xi'=\xi$ by conditions \ref{cond:topRowDominates} since the orbit of any generic point in $\Omega'$ would converge to $\xi$. Then, we may assume that $U$ is an open subset of $\Omega'$, so that $\Omega\subset\Omega_1\subset\Omega_0\subset\Omega'$, and since $\overline{\Omega}\cap\partial\Omega_1=\{\xi\}$, we must have $\overline{\Omega}\cap\partial\Omega'=\{\xi\}$.
\end{remark}

In the round case, we need to produce a domain that is $\COne$ at the fixed point of $\Gamma$ before applying the smoothing lemma.

\begin{proposition}\label{prop:suffCondRoundCase}
   Let $\Gamma$ be an unbounded subgroup of $\SLpm(n,\R)$ satisfying \ref{cond:topRightDominates}. There exists a $\Gamma$-invariant round domain $\Omega$ on which the action of $\Gamma$ is doubly elementary of the parabolic type fixing $\xi\in\partial\Omega$ and a hyperplane $H$ supporting $\Omega$ at $\xi$, such that $\partial\Omega-\{\xi\}$ is the graph of an analytic function in the affine chart $\P(\R^n)-H$.

   In particular, if $\Gamma$ is discrete, then $\Omega/\Gamma$ is a round cusp.
\end{proposition}

\begin{proof}
   We keep the notations given in condition \ref{cond:topRightDominates}.
   By Remark \ref{rmk:TReImpliesGP+}, \ref{cond:topRightDominates} implies \ref{cond:existsPosGenPGLC}, so $\Gamma$ satisfies \ref{cond:existsGenPLC} and \ref{cond:topRowDominates}. Therefore, we can apply Proposition \ref{prop:suffCondStrictlyConvCase} to get a $\Gamma$-invariant strictly convex domain $\Omega_1\subset\P(\R^n)$ on which the action of $\Gamma$ is doubly elementary of the parabolic type. Let $C_1$ be the $\Gamma$-invariant sharp open convex cone lifting $\Omega_1$ whose closure contains $p$.

   We choose a basis as required by condition \ref{cond:topRightDominates}, and we further assume that $e_n\in C_1$. Up to multiplying $\phi$ by a constant, we impose $\phi(e_n) = -1$, so that $\phi\in\partial C_1^\ast$. We choose a compact subset $K$ of $C_1^\ast$ such that $\P(K)$ has non-empty interior, and $\beta(e_1)=-1$ for all $\beta\in K$ (this is possible because $e_1=p$ does not lie in $C_1$).

   By condition \ref{cond:topRightDominates}, the sum $\gamma\mapsto\sum_{(i,j)\neq (1,n)}\abs{\gamma_{ij}}$ is dominated by $\gamma\mapsto \gamma_{1n}$, so we may choose a positive function $f\in\Cont(\Gamma)$ of intermediate order, meaning that $f$ is dominated by the top right entry of $\Gamma$, and dominates the sum of the absolute values of the other entries of $\Gamma$.

   For any fixed $m\in\N$, the function $\gamma\mapsto\gamma_{1n} - m^2f(\gamma)$ is s-proper and bounded below, as it is equivalent to the top right entry of $\Gamma$, which is s-proper by Remark \ref{rmk:TReImpliesTReProper}, and bounded below since $e_n\in C_1$, so that $\gamma e_n/\norm{\gamma e_n}\to p=e_1$, and $\gamma_{1n}=e_1^\ast(\gamma e_n)\sim \norm{\gamma e_n} e_1^\ast(e_1)\to +\infty$ for $\gamma\to\infty$ in $\Gamma$. Therefore we can choose $b_m\geq 0$ such that $\gamma_{1n}+b_m-m^2f(\gamma)\geq 0$ for all $\gamma\in\Gamma$. It follows that $(\gamma,m)\mapsto \gamma_{1n} + b_m$ dominates $(\gamma,m)\mapsto m\sum_{(i,j)\neq (1,n)}\abs{\gamma_{ij}}$.

   Consider the points $x_m^\epsilon = e_n + m\sum_{i=2}^{n-1}\epsilon_i e_i + b_m e_1$ for all $\epsilon\in\{\pm 1\}^{n-2}$ and $m\in\N$. For all $\gamma$, $m$ and $\epsilon$, we have:
   \begin{equation}\label{eq:accPointsOfOrbitsOfAddedPoints}
      \gamma [x_m^\epsilon] = \begin{bmatrix}
         \gamma_{1n} + m\sum_{i=2}^{n-1}\epsilon_i\gamma_{1i} + b_m \\
         \gamma_{2n} + m\sum_{i=2}^{n-1}\epsilon_i\gamma_{2i} \\
         \vdots \\
         \gamma_{n-1\; n} + m\sum_{i=2}^{n-1}\epsilon_i\gamma_{n-1\; i} \\
         1
      \end{bmatrix} 
      \tosub{\substack{(m,\gamma)\to \infty \\ (m,\gamma)\in\N\times\Gamma}} \xi
   \end{equation}

   Hence the function $(m,\gamma)\mapsto-\beta(\gamma x_m^\epsilon)$ is s-proper and bounded below for every $\beta\in K$. Moreover, $K$ is compact, so up to increasing $b_m$, we may suppose that $-\beta(\gamma\cdot x_m^\epsilon)\geq 0$ for all $m$, $\gamma$, $\epsilon$ and $\beta$. Therefore, by Lemma \ref{lm:sharpConvConeNSC},
   \begin{equation*}
      \Omega_2 = \Conv\left\{\Omega_1\cup\bigcup_{\substack{m\in\N \\\epsilon\in\{\pm 1\}^{n-2}}}\Gamma\cdot [x_m^\epsilon]\right\}
   \end{equation*}
   is a properly convex $\Gamma$-invariant domain, where we take the convex hull in the affine chart defined by $H$. Since $\Omega_1$ is strictly convex, $\overline{\Omega_1}\cap H=\{\xi\}$. Also, the closure of the union of the orbits of the points $[x_m^\epsilon]$ only meets $H$ at $\xi$ by Equation \eqref{eq:accPointsOfOrbitsOfAddedPoints}, so $\overline{\Omega_2}\cap H = \{\xi\}$. Let $C_2\supset C_1$ be an open convex cone lifting $\Omega_2$.

   We now check that $\Omega_2$ is $\COne$ at $\xi$. Let $\alpha\in(\R^n)^\ast-\{0\}$ such that $\P(\ker\alpha)$ supports $\Omega_2$ at $\xi$. We must have $\alpha(p) = 0$, \ie $\alpha_1 = 0$, and since $e_n\in C_1\subset C_2$, we may assume up to multiplying $\alpha$ by a scalar that $\alpha_n = -1$. Therefore, $\alpha$ takes negative values on $C_2$.

   In particular, $\alpha(x_m^\epsilon)\leq 0$ for all $m$ and $\epsilon$, that is
   \begin{equation*}
      -1 + m\sum_{i=2}^{n-1}\epsilon_i \alpha_i \leq 0
   \end{equation*}

   However, choosing $\epsilon_i$ to be the sign of $\alpha_i$ when it is non-zero, we see that unless $\alpha_i=0$ for all $i\in\{2,\dots,n-1\}$, the left hand side diverges to $+\infty$ when $m\to+\infty$, which is a contradiction. Therefore $\alpha=\phi$, so $\Omega_2$ is $\COne$ at $\xi$.

   We can now apply Lemma \ref{lm:smoothing} to get a $\Gamma$-invariant round domain $\Omega\subset\Omega_2$ such that $\partial\Omega-\{\xi\}$ is the graph of a strictly convex analytic function in the affine chart $\P(\R^n)-H$. We conclude that $\Gamma$ is a doubly elementary group of the parabolic type as in the last paragraph of the proof of Proposition \ref{prop:suffCondStrictlyConvCase}.
\end{proof}

\begin{remark}\label{rmk:TReImpliesRoundDomainInCOneDomain}
   Keeping the same notations as in the proof of the last proposition, it appears that if $\Gamma$ is already elementary of the parabolic type in a domain $\Omega'$ that is $\COne$ at the unique fixed point $\xi'$ of $\Gamma$, then we can produce a $\Gamma$-invariant round domain $\Omega\subset\Omega'$ such that $\overline{\Omega}\cap\partial\Omega'=\{\xi'\}$ in which $\Gamma$ is elementary of the parabolic type. 
   Indeed, observe first that $\xi'=\xi$ since the orbit of a generic point in $\Omega'$ would converge to $\xi$ by condition \ref{cond:topRightDominates}, and the same argument made in $\Omega'^\ast$ shows that the only supporting hyperplane of $\Omega'$ at $\xi$ is $H$. Observe also that by Remark \ref{rmk:TrImpliesStrictlyConvDomInInvariantDomain}, we may suppose that $\Omega_1\subset\Omega'$.

   Now, we only need to make sure that all the points $[x_m^\epsilon]$ lie in $\Omega'$, as it would imply that their $\Gamma$-orbit also does. Note that since $\xi$ is a $\COne$ point of $\Omega'$, and for all $m$ and $\epsilon$, $x_m^\epsilon\not\in H$, the projective line through $\xi$ and $[x_m^\epsilon]$ intersects $\Omega'$. Observe also that increasing $b_m$ amounts to translating each of the $[x_m^\epsilon]$ on the projective line through $\xi$ and $[x_m^\epsilon]$ in the direction of $\xi$, from the side of the interior of $\Omega'$ (that is because both $e_1$ and $e_n$ lie in the closure of a sharp open convex cone lifting $\Omega_1\subset\Omega'$). Therefore, we can always suppose up to increasing $b_m$ that $[x_m^\epsilon]\in\Omega'$.
\end{remark}

\section{Generalized cusps of non-maximal rank}

In this section, we build generalized cusps of non-maximal rank, as defined first by Cooper--Long--Tillmann.

\begin{definition}[{{\cite[\Def 6.1]{cooperLongTillmann2018}}}] 
   A \emph{generalized cusp} is a convex projective orbifold with boundary $C = \Omega/\Gamma$ with virtually nilpotent holonomy that is homeomorphic to $\partial C\times \R_{\geq 0}$, and such that $\partial C$ is strictly convex.

   If moreover $\partial C$ is compact, we say that $C$ has \emph{maximal rank}.
\end{definition}

Ballas, Cooper and Leitner classified the generalized cusps of maximal rank, and studied their deformation theory \cite{ballasCooperLeitner2020,ballasCooperLeitner2022}. In particular, they proved that generalized cusps of maximal rank always have abelian holonomy. This is similar to how in strictly convex manifolds, maximal rank cusps are equivalent to hyperbolic cusps (see \cite[\Thm 0.5]{cooperLongTillmann2015}).

It turns out however that solvability is a more natural assumption on the holonomy of projective cusps. Indeed, by the Tits alternative, linear groups are either virtually solvable or contain a free group, so the natural negation of having free subgroups is virtual solvability. Free groups should be avoided if generalized cusps are thought of as connected components of a sort of “generalized thick-thin decomposition”. Indeed, one typically expects the holonomy of the connected components of the thin part to be elementary, which is implied up to finite index by the assumption that the holonomy be virtually solvable by Proposition \ref{prop:solvGpsAreElem}, whereas free groups are typically non-elementary. 
We therefore weaken the requirement of virtual nilpotency in the above definition, to allow for virtual solvability.

\begin{definition}\label{def:genCusp}
   A generalized cusp is a convex projective orbifold with  boundary $C=\Omega/\Gamma$ with virtually solvable holonomy that is homeomorphic to $\partial C\times\R_{\geq 0}$, and such that $\partial C$ is strictly convex.
\end{definition}

We believe what was mentioned above justify that we modify the original definition of Cooper, Long and Tillmann to this one. From now on, every mention of generalized cusps refers to this definition. This allows for new examples, as we will see in paragraph \ref{par:genCuspWithSolvHolo}.

\subsection{A geometric description of generalized cusps} \label{par:genCuspGeom}

We start by giving a geometric framework inspired by Ballas--Cooper--Leitner's description of the generalized cusps of maximal rank for building generalized cusps.

\begin{definition}\label{def:genCuspDomain}
   A properly convex domain $\Omega\in\P(\R^n)$ is an \emph{analytic generalized cusp domain with parabolic face $\omega$} if
   \begin{enumerate}
      \item $\omega$ is an open face of $\partial\Omega$;
      \item there is no non-trivial segment in $\partial\Omega-\overline{\omega}$;
      \item[(2')] $\partial\Omega-\overline{\omega}$ is the graph of a strictly convex analytic function in some affine chart $\P(\R^n)-H$ for $H$ a supporting hyperplane of $\Omega$ containing $\omega$.
   \end{enumerate}
   If $\Omega$ satisfies the first two conditions only, we say that $\Omega$ is a \emph{generalized cusp domain}.

   If $\Omega$ is a generalized cusp domain and there is a unique hyperplane supporting $\Omega$ that contains $\omega$, we say that its parabolic face $\omega$ is $\COne$.
\end{definition}

For our purpose, these generalized cusp domains are only interesting insofar as they admit an action by a group such that the quotient is a generalized cusp.

\begin{definition}\label{def:genCuspGp}
   If $\Omega$ is a generalized cusp domain with parabolic face $\omega$, we say that a subgroup $\Gamma$ of the automorphisms of $\Omega$ acts as a \emph{generalized cusp group} if
   \begin{enumerate}
      \item $\Gamma$ preserves $\omega$ and its restriction to $\Span\omega$ is virtually solvable;
      \item there is a point $\xi$ of $\overline{\omega}$ and a hyperplane $H\supset\omega$ supporting $\Omega$ such that $G$ preserves the algebraic horospheres of $\Omega$ centered at $(\xi,H)$.
   \end{enumerate}
\end{definition}

The geometric characterization of generalized cusps we are aiming at can be summarized as in the following definition.

\begin{definition}
   Let $\Omega\subset\P(\R^n)$ be a generalized cusp domain with parabolic face $\omega$ on which a discrete group $\Gamma$ acts as a generalized cusp group. We call the quotient $\Omega/\Gamma$ a \emph{geometric generalized cusp}.

   If moreover $\Omega$ is an analytic generalized cusp domain, we say that $\Omega/\Gamma$ is an \emph{analytic geometric generalized cusp}.
\end{definition}

We now prove that these definitions really allow to build generalized cusps. The following lemma will be used in the proof. We are grateful to Konstantinos Tsouvalas who gave us a proof of a similar statement, which we adapted to this case.

\begin{lemma}\label{lm:extensionsOfVirtSolvAreVirtSolv}
   Let $1\to A\to B\to C\to 1$ be a short exact sequence where $A$ and $C$ are virtually solvable. Then $B$ is virtually solvable.
\end{lemma}
\begin{proof}
   Let $C'$ be a finite index normal solvable subgroup of $C$, that we can produce by intersecting the finitely many conjugates of a finite index solvable subgroup of $C$. Since $C\simeq B/A$, there is a normal subgroup $B'$ of $B$ such that $A$ is normal in $B'$ and $C'\simeq B'/A$. Moreover, $B/B'\simeq (B/A)/(B'/A)\simeq C/C'$ is finite, so $B'$ has finite index in $B$.

   Consider $A'$ a finite index characteristic solvable subgroup of $A$ (which we can construct in a similar way as $C'$ above). Since $A'$ is characteristic in $A$ which is normal in $B'$, $A'$ is a normal subgroup of $B'$. Moreover, since $C'\simeq B'/A\simeq (B'/A')/(A/A')$, we see that we have a short exact sequence $1\to A/A'\to B'/A'\to C'\to 1$, where $A/A'$ is finite and $C'$ is solvable.

   Let $N$ be the centralizer of $A/A'$ in $B'/A'$. It has finite index in $B'/A'$ as it is the kernel of the morphism $x\in B'/A'\mapsto (y\in A/A'\mapsto xyx^{-1})\in\Aut(A/A')$ whose target is a finite group. Moreover, the group $G$ generated by $N$ and $A/A'$ in $B'/A'$ is the direct product $N\times (A/A')$ since $N$ acts trivially on $A/A'$ by conjugation. Therefore, the projection of $G$ into $C$ is isomorphic to $(N\times (A/A'))/(A/A') \simeq N$. Since $C$ is solvable, it follows that $N$ is solvable.

   If now $B''$ is the subgroup of $B'$ such that $B''/A'\simeq N$, we see that $B''$ is solvable since both $A'$ and $B''/A'$ are solvable. Note moreover that $B''$ is normal in $B'$ since $N$ is normal in $B'/A'$, and that $B'/B''\simeq (B'/A')/(B''/A')\simeq (B'/A')/N$ is finite. Hence $B''$ is a finite index solvable subgroup of $B'$, which itself has finite index in $B$. So $B$ is virtually solvable.
\end{proof}

We can now prove that geometric generalized cusps are generalized cusps.

\begin{proposition}\label{prop:geomGenCuspsAreGenCusps}
   Let $C$ be a geometric generalized cusp. Then its closure $\overline{C}$ is a generalized cusp.
\end{proposition}
\begin{proof}
   Let $\Omega\subset\P(\R^n)$ be a generalized cusp domain with parabolic face $\omega$ and $\Gamma<\Aut(\Omega)$ be a group acting as a generalized cusp group on $\Omega$, so that $C=\Omega/\Gamma$. Then $\overline{C}=(\overline{\Omega}-\omega)/\Gamma$. Since $\Omega$ is a generalized cusp domain, $\partial\Omega-\omega$ does not contain non-trivial segments, so that $\partial C = (\partial\Omega-\omega)/\Gamma$ is strictly convex.

   Let $(\xi,H)$ be the center of the algebraic horospheres preserved by $G$. 
   
   Let $\gamma\in\Gamma$ be a hyperbolic element. It fixes points $x_\gamma^+,x_\gamma^-\in\partial\Omega$ that are respectively eigenvectors for its largest and smallest eigenvalues. Since $\partial\Omega-\omega$ contains no non-trivial segments, if at least one of these two points lies outside of $\overline{\omega}$, then $(x_\gamma^-,x_\gamma^+)\subset\Omega$, and $\gamma$ cannot preserve the $(\xi,H)$-algebraic horospheres. Hence they both lie in $\overline{\omega}$, and $\gamma$ must act non-trivially on $\omega$.
   It follows that the subgroup $N$ of the elements of $\Gamma$ which act trivially on $\omega$ does not contain any hyperbolic element. So by \cite[\Prop 4.11]{cooperLongTillmann2015}, $N$ is virtually nilpotent.

   Moreover, $\Gamma/N$ identifies with the restriction of $\Gamma$ to $\Span\omega$, so it is virtually solvable by assumption. Hence $\Gamma$ is virtually solvable by Lemma \ref{lm:extensionsOfVirtSolvAreVirtSolv}.

   Remains to prove that $\overline{C}$ is homeomorphic to $\partial C\times \R_{\geq 0}$. Since $\partial\Omega-\omega$ does not contain any non-trivial segment, the foliation of $\Omega$ by the horospheres centered at $(\xi,H)$ extends to $\overline{\Omega}-\omega$, and we have
   \begin{equation*}
      \overline{\Omega}-\omega = \bigsqcup_{t\geq 0}\horo_t
   \end{equation*}
   where $\horo_t = \phi_t^{(\xi,H)}(\partial\Omega-\omega)$. Since $\Gamma$ preserves these horospheres, we have by Lemma \ref{lm:autThatPreservesAlgHoroCommutesWithFlow}
   \begin{align*}
      \overline{C} &= (\overline{\Omega}-\omega)/\Gamma \\
      &=\bigsqcup_{t\geq 0}\phi_t^{(\xi,H)}\left((\partial\Omega-\omega)/\Gamma\right) \\
      &=\bigsqcup_{t\geq 0}\phi_t^{(\xi,H)}\partial C
   \end{align*}
   so that
   \begin{equation*}
      h:\left\{\begin{aligned}
         \partial C\times\R_{\geq 0} &\to \overline{C} \\
         ([x],t) &\mapsto \phi_t^{(\xi,H)}(x)
      \end{aligned}\right.
   \end{equation*}
   is a well-defined homeomorphism between $\partial C\times\R_{\geq 0}$ and $\overline{C}$ (it is clearly bijective by the last equation, and continuous and open by definition of $\phi_t^{(\xi,H)}$).
\end{proof}

\subsection{Holonomies of generalized cusps}

In \cite{ballasCooperLeitner2020}, Ballas, Cooper and Leitner show that a generalized cusp of dimension $n$ is a quotient $\Omega/\Gamma$, where $\Gamma<\PGL(n+1,\R)$ is a lattice in a certain Lie group $G$ isomorphic to a subgroup of Euclidean isometries, and $\Omega\subset\P(\R^{n+1})$ is a properly convex domain determined by $G$ (as the convex hull of a generic orbit of $G$). See Theorems 0.2 and 1.45 and Equation (1.33) in \cite{ballasCooperLeitner2020} for more details about this description.

$G$ is uniquely determined up to conjugation by the data of an integer $s\in\{0,\dots,n-1\}$ and of a point $[\psi]\in\P(\R^s)^\ast$ such that the associated affine chart of $\P(\R^s)$ contains the simplex defined by the standard basis of $\R^s$. More precisely, up to a compact factor, $G$ is the subgroup of $\PGL(n+1,\R)$ given by:
\begin{equation*}
   \left\{\begin{pmatrix}
      e^{\diag{X}} & 0 \\
      0 & \rho(Y)\circ\Phi_\psi(X)
   \end{pmatrix}, (X,Y)\in\R^s\times \R^{n-1-s}\right\}
\end{equation*}
where
\begin{equation*}
   \rho(Y) = \begin{pmatrix}
      1 & \transpose{Y} & \frac12 \norm{Y}^2 \\
      0 & I_{n-1-s} & Y \\
      0 & 0 & 1
   \end{pmatrix}
\end{equation*}
is the classical isomorphism between $\R^{n-1-s}$ and the full group of translations on a horosphere of $\H^{n-s}$, and
\begin{equation*}
   \Phi_\psi(X) = \begin{pmatrix}
      0 & \cdots & 0 & -\psi(X) \\
      0 & \cdots & 0 & 0 \\
      \vdots & \ddots & \vdots & \vdots \\
      0 & \cdots & 0 & 0
   \end{pmatrix}
\end{equation*}
is a representation that commutes with $\rho$, and intermingles the action of $X$ and $Y$.

The idea is to use the same description of groups as presented above, and replace the representation $\rho$ which whose restriction to a lattice of $\R^{n+s-1}$ is the holonomy of a hyperbolic cusp with any representation satisfying \ref{cond:existsGenPLC} and \ref{cond:topRowDominates} (in particular, it may be interesting if possible to apply this construction to well-behaved syndetic hulls of the holonomies of strictly convex cusps).

Let $\rho:N\to\SLpm(n,\R)$ be a representation satisfying \ref{cond:existsGenPLC} and \ref{cond:topRowDominates}, $s\in\N$ and $[\psi]\in\P(\R^s)$. We may suppose up to conjugating $\rho$ that it is of the form
\begin{equation*}
   \begin{pmatrix}
      1 & * & *\\
      0 & * & * \\
      0 & 0 & 1
   \end{pmatrix}
\end{equation*}
Therefore, $\rho$ and $\Phi_\psi$ commute, and
\begin{equation*}
   \genRep{\rho}{\psi}: (X,Y)\in\R^s\times N \mapsto \begin{pmatrix}
      e^{\diag{X}} & 0 \\
      0 & \rho(Y)\circ\Phi_\psi(X)
   \end{pmatrix}
\end{equation*}
is a well-defined representation $\R^s\times N\to\GL(n+s,\R)$. 

We now investigate using Proposition \ref{prop:gpPreservesDomainNSC} under what conditions $\genRep{\rho}{\psi}$ preserves a properly convex domain.

\begin{proposition}\label{prop:genRepPreservesDomainNSC}
   $\genRep{\rho}{\psi}$ preserves a properly convex domain if and only if all coordinates of $\psi$ have the same sign $\sigma\in\{\pm 1\}$ and there is a positive s-proper generic matrix coefficient of $\rho$ whose top-right coefficient has the sign $\sigma$.

   In particular, $\genRep{\rho}{0}$ always preserves a properly convex domain.
\end{proposition}
\begin{proof}
   We start by projecting the image of $\genRep{\rho}{\psi}$ to $\SLpm(n+s,\R)$ so that we can use Proposition \ref{prop:gpPreservesDomainNSC}. Since $\rho$ is a representation into $\SLpm(n,\R)$, $\abs*{\det \genRep{\rho}{\psi}(X,Y)}= \exp{\sum_i X_i}$, so that
   \begin{equation*}
      \mathcal{D}\genRep{\rho}{\psi}(X,Y) = \exp\left(-\frac1{n+s}\sum_iX_i\right)\genRep{\rho}{\psi}(X,Y)
   \end{equation*}
   is the projection into $\SLpm(n+s,\R)$.

   We start by supposing there is a positive generic s-proper matrix coefficient of $\mathcal{D}\genRep{\rho}{\psi}$. We may write it as
   \begin{equation*}
      P:(X,g)\in\R^s\times N \mapsto e^{-\frac1{n+s}\sum_iX_i}\left(\sum_i\mu_ie^{X_i} - \mu_{1n}\psi(X) + \sum_{ij}\mu_{ij}(\rho(g))_{ij}\right)
   \end{equation*}
   where $\mu_i$ is indexed from $1$ to $s$ and $\mu_{ij}$ is indexed on $\{1,\dots,n\}^2$ and is a matrix coefficient of $\rho$. We can decouple this expression by denoting
   \begin{equation*}
      Q:g\mapsto \sum_{ij}\mu_{ij}(\rho(g))_{ij}
   \end{equation*}
   and
   \begin{equation*}
      T:X\mapsto e^{-\frac1{n+s}\sum_iX_i}\left(\sum_i\mu_ie^{X_i}-\mu_{1n}\psi(X)\right)
   \end{equation*}
   so that
   \begin{equation}\label{eq:decoupledPLC}
      P(X,g) = T(X) + e^{-\frac1{n+s}\sum_iX_i}Q(g)
   \end{equation}

   Since $P$ is a positive generic s-proper linear combination of the entries of $\mathcal{D}\genRep{\rho}{\psi}$, it is clear that $Q$ must also be s-proper and bounded below by setting $X=0$ in \eqref{eq:decoupledPLC}. Hence $Q$ is a generic s-proper bounded below matrix coefficient of $\rho$.
   
   In the case where $\mu_{1n}\psi_i\neq 0$, we have
   \begin{align*}
      T(0,\dots,X_i,\dots,0) &= e^{-\frac1{n+s}X_i}(\mu_ie^{X_i}-\mu_{1n}\psi_iX_i + O(1)) \\
      &\simsub{X_i\to -\infty} \mu_{1n}\psi_i\abs{X_i}e^{\frac1{n+s}\abs{X_i}} \\
      &\tosub{X_i\to -\infty} \mu_{1n}\psi_i\infty
   \end{align*}
   so $\mu_{1n}\psi_i>0$. Therefore, we always have $\mu_{1n}\psi_i\geq 0$, so all the coordinates of $\psi$ have the same sign $\epsilon\in\{\pm 1\}$, and there exists a generic s-proper bounded below linear combination of the entries of $\rho$ whose top-right coefficient has sign $\epsilon$.

   Observe moreover that for all $i\in\{1,\dots, s\}$, then in all cases, we have (with the convention that $0\cdot\infty = 0$)
   \begin{align*}
      T(0,\dots,X_i,\dots,0) &= e^{-\frac1{n+s}X_i}(\mu_ie^{X_i}-\mu_{1n}\psi_iX_i + O(1)) \\
      &\tosub{X_i\to +\infty} \mu_i\infty
   \end{align*}
   Since $P$ is s-proper and bounded below, it follows that $\mu_i>0$. 

   \medskip

   Now suppose that there is a generic s-proper bounded below matrix coefficient of $\rho$ $(\mu_{ij})_{ij}$ with $\mu_{1n}$ of sign $\sigma\in\{\pm 1\}$, and choose any $\psi$ whose coordinates all have the sign of $\sigma$. Since $(\mu_{ij})_{ij}$ is generic, we may suppose that $\mu_{1n}\neq 0$, and up to multiplying $(\mu_{ij})_{ij}$ by a positive scalar, we may assume that $\mu_{1n} >1$. Moreover, we may assume by increasing $\mu_{11}$ as in Remark \ref{rmk:TReImpliesGP+} that $Q(g) = \sum_{ij}\mu_{ij}(\rho(g))_{ij}>1$ for all $g\in N$, and that this remains true for small perturbations of $(\mu_{ij})_{ij}$. For all $i$, choose $\mu_i > \abs{\psi_i}+\frac12$.

   Observe that since $e^x\geq 1+ x$ for all $x\in\R$, we have for any $X\in\R^s$
   \begin{align*}
      \sum_i\mu_ie^{X_i} - \mu_{1n}\psi(X) &\geq \sum_i(\mu_i e^{X_i} - \abs{\mu_{1n}}\abs{\psi_i}e^{X_i}) \\
      &\geq \sum_i(\mu_i-\abs{\psi_i})e^{X_i} \\
      &\geq \frac12\sum_i e^{X_i}
   \end{align*}
   so that for all $(X,g)\in\R^s\times N$, and for any small perturbation as a matrix coefficient of $\mathcal{D}\genRep{\rho}{\psi}$
   \begin{equation}\label{eq:ineqGenRepSProper}
      \begin{aligned}
         P(X,g) &= e^{-\frac1{n+s}\sum_iX_i}\sum_i(\mu_ie^{X_i}-\mu_{1n}\psi_iX_i) + e^{-\frac1{n+s}\sum_iX_i}Q(g) \\
         &\geq \frac12 \sum_i e^{X_i-\frac1{n+s}\sum_iX_i} + e^{-\frac1{n+s}\sum_iX_i}Q(g)
      \end{aligned}
   \end{equation}
   is a positive matrix coefficient of $\mathcal{D}\genRep{\rho}{\psi}$. We claim that it is also s-proper and generic.
    Indeed, if $(X,g)\to\infty$ in $\R^s\times N$, either $X$ remains bounded, or $X\to\infty$ in $\R^s$. In the first case, we must have $g\to\infty$ in $N$, and we conclude that $P(X,g)\to +\infty$ from the fact that $g\mapsto Q(g)$ is s-proper and bounded below. In the second case, observe that since $Q(g)\geq 1$ for any $g\in N$, we just need to show that
    \begin{equation*}
       X\mapsto\frac12 \sum_i e^{X_i-\frac1{n+s}\sum_iX_i} + e^{-\frac1{n+s}\sum_iX_i} 
    \end{equation*}
    is s-proper and bounded below on $\R^s$. For this, note that letting $Y_i=X_i-\frac1{n+s}\sum_i X_i$ is a linear change of variable, so that $X\to\infty$ in the variables $X_1,\dots,X_s$ if and only if it does so in the variables $Y_1,\dots,Y_s$. Moreover, if any one of the $Y_i$ goes to $+\infty$, then the term $\exp{Y_i}$ goes to $+\infty$, and so does $P(X,g)$. Otherwise, all the $Y_i$ are bounded above, so $\sum_i Y_i = \frac n{n+s}\sum_i X_i$ is also bounded above. But since $X\to\infty$ in $\R^s$, we must have $\min Y_i\to -\infty$, so that $\sum_i X_i=\frac{n+s}n\sum_i Y_i\to -\infty$ and $\exp(-\frac1{n+s}\sum_i X_i)\to +\infty$. Hence the function above is positive and s-proper on $\R^s$.

    We conclude that $(X,g)\mapsto P(X,g)$ is a positive s-proper matrix coefficient of $\mathcal{D}\genRep{\rho}{\psi}$ that remains so under any small perturbation of its coefficients, and is therefore generic, since \eqref{eq:ineqGenRepSProper} holds for small perturbations as well. By Proposition \ref{prop:gpPreservesDomainNSC}, $\mathcal{D}\genRep{\rho}{\psi}$ preserves a properly convex domain.

\end{proof}

\begin{remark}
   Since $\rho$ satisfies \ref{cond:existsPosGenPGLC} and \ref{cond:topRowDominates}, we can always suppose by Remark \ref{rmk:TReImpliesGP+} that there is a positive generic s-proper matrix coefficient of $\rho$ with \emph{positive coefficients}, up to changing the sign of the elements $e_{2},\dots,e_{n}$ of the basis of $\R^n$.

   From now on, we will assume that all the representations we consider satisfy this. In particular, there are always positive generic s-proper matrix coefficients of $\rho$ with positive top-right coefficient, so the above proposition shows that $\genRep{\rho}{\psi}$ always preserves a properly convex domain when $\psi$ have non-negative coordinates.

\end{remark}

\subsection{Generalized cusps with virtually nilpotent holonomy}

We now proceed with building domains that are stabilized by $\genRep{\rho}{\psi}$. We follow in that regard the construction given in \cite{ballasCooperLeitner2020}. In particular, we will require, even though this is not necessary according to Proposition \ref{prop:genRepPreservesDomainNSC}, that all coordinates of $\psi$ are strictly positive.

Let $\rho:\Gamma\to\SLpm(n,\R)$ be a representation satisfying \ref{cond:existsGenPLC} and \ref{cond:topRowDominates}. By Theorem \ref{thm:characHoloStrictlyConvCusps}, there is a $\Gamma$-invariant strictly convex domain $\Omega\subset\P(\R^n)$ with a single $\Gamma$-fixed point $\xi\in\partial\Omega$, as well as a $\Gamma$-fixed hyperplane supporting $\Omega$ at $\xi$, such that $\partial\Omega-\{\xi\}$ is the graph of a strictly convex analytic function on $\P(\R^n)-H$. If moreover $\Gamma$ satisfies \ref{cond:topRightDominates}, then $\Omega$ can be chosen to be round by Theorem \ref{thm:characHoloRoundCusps}, that is, we can find $\Omega$ as above such that $H$ is the only hyperplane supporting $\Omega$ at $\xi$.

Recall that $\genRep{\rho}{\psi}$ acts on $\P(\R^s\oplus\R^n)$, where we have chosen a basis $(e_{s+1},\dots,e_{s+n})$ of $\R^n$ such that there is a positive generic s-proper matrix coefficient of $\rho$ with positive coefficients. We can moreover assume by Remark \ref{rmk:TrTReMatrixForm} that $\xi = [e_{s+1}]$ and that $H = \P(\gen{e_{s+1},\dots,e_{s+n-1}})$.

Consider $H' = \P(\gen{e_1,\dots,e_{s+n-1}})$ the hyperplane spanned by $\P(\R^s)$ and $H\subset\P(\R^n)$, and $\A = \P(\R^s\oplus\R^n)-H'$ the associated affine chart. We identify $\A$ with $\R^s\oplus\R^{n-2}\oplus\R$ via the following isomorphism
\begin{equation*}
   \begin{cases}
      \P(\R^s\oplus\R^n)- H' &\to \R^s\oplus\R^{n-2}\oplus\R\\
      [U:h:V:1]&\mapsto (U,V,h)
   \end{cases}
\end{equation*}
where $\R^n=\gen{e_{s+1},\dots,e_{s+n}}$ is decomposed as $\R\oplus\R^{n-2}\oplus\R = \gen{e_{s+1}}\oplus\gen{e_{s+2},\allowbreak\dots,\allowbreak e_{s+n-1}}\oplus\gen{e_{s+n}}$. Observe that in these coordinates, $0$ is identified with the point $[e_{s+n}]\in\P(\R^s\oplus\R^n)-H'$. In restriction to $\P(\R^n)$, we let $\phi:\D_\Omega\to\R$ be the strictly convex analytic function whose graph $\{(0,V,\phi(V)),V\in\D_\Omega\}$ is $\partial\Omega-\{\xi\}$.

In these coordinates, $\genRep{\rho}{\psi}$ acts following the formula, given any $X\in\R^s$ and $\gamma\in\Gamma$:
\begin{equation}\label{eq:genRepFormulaParabCoord}
   \begin{aligned}
      \genRep{\rho}{\psi}(X,\gamma)\cdot (U,V,h) &\simeq \genRep{\rho}{\psi}(X,\gamma)\begin{bmatrix}
         U\\ h \\ V \\ 1
      \end{bmatrix}
      =\begin{bmatrix}
         e^{\diag X}U \\ h_\gamma - \psi(X) \\ V_\gamma \\1
      \end{bmatrix}\\
      &\simeq (e^{\diag X}U,V_\gamma,h_\gamma-\psi(X))
   \end{aligned}
\end{equation}
where $(V_\gamma,h_\gamma)$ is determined by
\begin{equation*}
   \rho(\gamma)\begin{pmatrix}h\\ V\\ 1\end{pmatrix} = \begin{pmatrix}h_\gamma \\ V_\gamma \\ 1\end{pmatrix}
\end{equation*}

Following \cite[\Def 1.3]{ballasCooperLeitner2020}, the above data allows to define a horofunction, as well as the associated horospheres and properly convex domains.
\begin{definition}\label{def:genRepDomain}
   Let $V_{\psi,\Omega} := \R_{>0}^s\times\D_\Omega\times\R$. We define the \emph{$(\psi,\Omega)$-horofunction}
\begin{equation*}
   h_{\psi,\Omega}:\begin{cases}
      V_{\psi,\Omega}&\to \R\\
      (U,V,t) &\mapsto \phi(V) - \sum_{i=1}^s\psi_i\log U_i - t
   \end{cases}
\end{equation*}
   For any $t\in\R$, the level set $\horo_t:=h_{\psi,\Omega}^{-1}(t)$ of $h_{\psi,\Omega}$ is a \emph{$(\psi,\Omega)$-horosphere}. We denote by $\Omega_{\psi,\Omega}:=h_{\psi,\Omega}^{-1}(\R_{<0})$ the \emph{$(\psi,\Omega)$-domains}.
\end{definition}

\begin{remark}
   Observe that if, as in \cite{ballasCooperLeitner2020}, we had $\rho:N\to\PO(n-1,1)$ and $\Omega = \H^{n-1}$, Definition \ref{def:genRepDomain} would have defined exactly the same objects as \cite[\Def 1.3]{ballasCooperLeitner2020}, so this construction is a proper generalization of that of \cite{ballasCooperLeitner2020}.
\end{remark}

\begin{proposition}\label{prop:genRepDomainProperties}
   We have
   \begin{enumerate}
      \item $h_{\psi,\Omega}$ is an analytic $\genRep\rho\psi$-invariant function;
      \item the $(\psi,\Omega)$-horospheres define a foliation of $V_{\psi,\Omega}$ by $\genRep\rho\psi$-invariant strictly convex analytic hypersurfaces;
      \item the $(\psi,\Omega)$-horospheres limit to the $s$-simplex $\Delta:=\P(\R_{>0}e_1+\dots+\R_{>0}e_{s+1})$;
      \item $\Omega_{\psi,\Omega}\subset\P(\R^s\oplus\R^n)$ is a properly convex, $\genRep\rho\psi$-invariant domain, foliated by the $(\psi,\Omega)$-horospheres $(\horo_s)_{s>0}$;
      \item its boundary $\partial\Omega_{\psi,\Omega}$ is the union of $\horo_0$ and $\Delta$;
      \item $\Delta$ is a $\COne$ face of $\partial\Omega_{\psi,\Omega}$ if and only if $\Omega$ is round.
   \end{enumerate}
\end{proposition}
\begin{proof}
   \begin{enumerate}
      \item Since $\phi$ is analytic on its domain, $h_{\psi,\Omega}$ clearly is as well. Moreover, we compute with the help of Equation \eqref{eq:genRepFormulaParabCoord}, for any $(X,\gamma)\in\R^s\times N$:
         \begin{align*}
            h_{\psi,\Omega}(\genRep\rho\psi(X,\gamma)\cdot (U,V,t))&=h_{\psi,\Omega}(e^{\diag X}U,V_\gamma,t_\gamma-\psi(X))\\
            &= \phi(V_\gamma)-\sum_i\psi_i\log(e^{X_i}U_i) - (t_\gamma-\psi(X)) \\
            &= \phi(V)-\sum_i\psi_i\log U_i -t = h_{\psi,\Omega}(U,V,t)
         \end{align*}
         where the identity $\phi(V_\gamma)-t_\gamma = \phi(V)-t$ comes from the $\rho$-invariance of $\partial\Omega-\xi$ seen as the graph of $\phi$ in the $\rho$-invariant affine chart $\P(\R^n)-H$.
      \item By the last item, the $(\psi,\Omega)$-horospheres define a foliation by $\genRep\rho\psi$-invariant analytic hypersurfaces. Observe that since $\horo_s = h_{\psi,\Omega}^{-1}(s)$ is the graph of the function $(U,V)\in\R_{>0}^s\times\D_\Omega \mapsto f(U,V) - s$, where $f(U,V) := \phi(V) - \sum_i\psi_i\log U_i$, it is enough to compute the order two Taylor expansion of $f$. We have, for any $(U,V)\in\R_{>0}^s\times\D_\Omega$ and $(u,v)\in\R^s\times\R^{n-2}$ small enough that $(U+u,V+v)\in\R_{>0}^s\times\D_\Omega$:
         \begin{multline*}
            f(U+u,V+v) = f(U,V) + D_V\phi(v) - \sum_i\psi_i\frac{u_i}{U_i}
            + \frac12 D_V^2\phi(v) +\sum_i\psi_i\left(\frac{u_i}{U_i}\right)^2 \\+ o(\norm{(u,v)}^2)
         \end{multline*}
         Since the quadratic term $(u,v)\mapsto \frac12 D_V^2\phi(v) + \sum_i\psi_i(\tfrac{u_i}{U_i})^2$ is clearly definite positive for any choice of $(U,V)$, as $\phi$ is a strictly convex analytic function and $\psi_i>0$ for all $i$, the graph of $f$ (and thus of $f-s$, for any constant $s\in\R$) is strictly convex.

         \item Since $\Omega$ is strictly convex, $\Omega\cap H' = \Omega\cap H = \{\xi\}$, so $\partial\Omega-\{\xi\}$, which is the graph of $\phi$, limits to the point $\xi$. In other words,
         \begin{equation*}
            [0:\phi(V):V:1]\tosub{\substack{V\to\infty \\ V\in \D_\Omega}} \xi=[e_{s+1}]
         \end{equation*}
         so that $\norm{V}=o(\phi(V))$ as $V$ diverges in $\D_\Omega$. In particular,

         \begin{equation*}
            N(U,V):= \norm{U} + \abs{\phi(V)} + \sum_i\max\{0,-\log U_i\} \to +\infty
         \end{equation*}
         if and only if $(U,V)\in\R_{>0}^s\times\D_\Omega$ diverges. Moreover, when $\norm{U}\leq \phi(V)$:
         \begin{multline*}
            \norm{U}+\abs{\phi(V)-\sum_i\psi_i\log U_i} \geq \norm{U} + \phi(V) - \sum_i\psi_i\log U_i \\
            \begin{aligned}
               &\geq \norm{U}+\phi(V) + \sum_i\psi_i (\max\{0,-\log U_i\} -\abs{\log \phi(V)}) \\
               &\geq \norm{U} + \phi(V) + \min_i\{\psi_i\}\sum\max\{0,-\log U_i\} - n\max_i\{\psi_i\}\abs{\log \psi(V)} \\
               &\geq \min_i\left\{\frac12,\psi_i\right\} N(U,V) - A
            \end{aligned}
         \end{multline*}
         for some constant $A>0$. Similarly, when $\norm{U}\geq \phi(V)$:
         \begin{multline*}
            \norm{U}+\abs{\phi(V)-\sum_i\psi_i\log U_i} \geq \norm{U} + \phi(V) - \sum_i\psi_i\log U_i \\
            \begin{aligned}
               &\begin{multlined}\geq \norm{U} + \phi(V) + \min_i\{\psi_i\}\sum_i\max\{0,-\log U_i\} \phantom{- n \max_i\{\psi_i\}\log \norm{U}} \\- n \max_i\{\psi_i\}\sum_i\max\{0,\log U_i\}\end{multlined} \\
               &\geq \norm{U} + \psi(V) + \min_i\{\psi_i\}\sum_i\max\{0,-\log U_i\} - n \max_i\{\psi_i\}\log \norm{U} \\
               &\geq \min_i\left\{\frac12,\psi_i\right\} N(U,V) - A
            \end{aligned}
         \end{multline*}
         up to increasing $A$. Since we also have
         \begin{multline*}
            \norm{U} + \abs{\phi(V)-\sum_i\psi_i\log U_i} \leq \norm{U} + \abs{\phi(V)} + \sum_i\psi_i\abs{\log U_i} \\
            \begin{aligned}
               &\leq \norm{U} + \abs{\phi(V)} +\sum_i\psi_i(\max\{0,-\log U_i\} + \max\{0,\log U_i\}) \\
               &\leq \norm{U} + \abs{\phi(V)} + \max_i\{\psi_i\}\sum_i\max\{0,-\log U_i\} + \max_i\{\psi_i\}\norm{U} \\
               &\leq \max_i\{1,\psi_i\} N(U,V)
            \end{aligned}
         \end{multline*}
         it follows that $\norm{U} + \abs{\phi(V) - \sum_i\psi_i\log U_i}$ and $N(U,V)$ have the same order when $(U,V)$ diverges in $\R_{>0}^s\times\D_\Omega$.

         Since $\norm{V}=o(\phi(V))$ as $V\to\infty$ in $\D_\Omega$, it follows that $\norm{V} = o(N(U,V))$ as $(U,V)\to\infty$ in $\R_{>0}^s\times\D_\Omega$. From these inequalities, we conclude that for all $s\in\R$, $[U:\phi(V)-\sum_i\psi_i\log U_i -s: V:1]$ converges to a point of the simplex $\Delta$ whenever $(U,V)$ diverges in $\R_{>0}\times\D_\Omega$. Hence the $(\psi,\Omega)$-horospheres limit to a point of $\Delta$.
         \item This follows from the strict convexity and the $\genRep\rho\psi$-invariance of $\horo_0$, and the fact that $\Omega_{\psi,\Omega}$ is contained in $\P(\R^s\oplus\R^n)-H'$. 
         \item This follows from the strict convexity of $\horo_0$ and the fact that it limits on $\Delta$.
         \item $\Delta$ is $\COne$ if and only if the only supporting hyperplane of $\Omega_{\psi,\Omega}$ containing $\Delta$ is $H'$. Since $\R^s$ is the span of $\Delta$, this holds if and only if the domain of $h_{\psi,\Omega}$ is $\R_{>0}^s\times\R^{n-2}$, if and only if $\D_\Omega = \R^{n-2}$, if and only if $\xi$ is a $\COne$ point of $\Omega$, if and only $\Omega$ is round (by assumption on $\Omega$).\qedhere
   \end{enumerate}
\end{proof}

\begin{corollary}\label{cor:genCuspDomainsForVirtNilpGps}
   Let $\rho:N\to\PGL(n,\R)$ be a representation satisfying \ref{cond:existsGenPLC} and \ref{cond:topRowDominates}, $\Omega\subset\P(\R^n)$ be a strictly convex domain with analytical boundary (apart from one point), $s$ be a non-negative integer, and $\psi\in(\R^s)^\ast$ be a linear form with positive coordinates. Then, $\Omega_{\psi,\Omega}$ is a $\genRep\rho\psi$-invariant analytic generalized cusp domain with parabolic face a $s$-simplex $\Delta$, and for any discrete subgroup $G$ of $N\times \R^s$, $\genRep\rho\psi(G)$ acts on $\Omega_{\psi,\Omega}$ as a generalized cusp group.

   Moreover, $\Delta$ is $\COne$ if and only if $\Omega$ is round.
\end{corollary}
\begin{proof}
   Proposition \ref{prop:genRepDomainProperties} proves almost everything needed. The only thing left to check is that the $\genRep\rho\psi$ preserves algebraic horospheres in $\Omega_{\psi,\Omega}$. Since $\genRep\rho\psi$ fixes both $\xi:=[e_{s+1}]\in\Delta$ and $H'=\P(\gen{e_1,\dots,e_{s+n-1}})\supset\Delta$ with the same eigenvalue, this is clear. In fact, it is easy to check that the $(\xi,H')$-algebraic horospheres of $\Omega_{\psi,\Omega}$ are exactly the $(\psi,\Omega)$-horospheres $(\horo_t)_{t\in\R}$, which are invariant by Proposition \ref{prop:genRepDomainProperties}.
\end{proof}

\subsection{A generalized cusp with solvable holonomy}\label{par:genCuspWithSolvHolo}

In this section, we build an example of a generalized cusp whose holonomy is isomorphic to a solvable non virtually nilpotent group. The holonomy is isomorphic to
\begin{equation*}
   R = \left\{\begin{pmatrix}
      \lambda^n & a+b\sqrt{2} \\
      0 & 1
   \end{pmatrix},a,b,n\in\Z\right\}
\end{equation*}
where $\lambda=(1+\sqrt{2})^2$. This group is isomorphic to the full upper triangular subgroup of $\SL(2,\Z[\sqrt{2}])$, and is a solvable non virtually nilpotent group.

We start by explaining how to build a discrete and faithful representation of $R$ that may act as a generalized cusp group on some generalized cusp domain. The first observation is that it should preserve a properly convex domain, have discrete image, and preserve a subspace on which the action is virtually solvable. Notice that there is no necessity a priori that it preserves a simplex on which it acts diagonally, but we will require this for simplicity.

We start by computing the action of $R$ on $\H^2$ seen as a subspace of the projectivization of the second symmetric power of $\R^2$:
\begin{equation*}
   \begin{pmatrix}
      \lambda^{2n} & \lambda^n(a+b\sqrt{2}) & (a+b\sqrt{2})^2 \\
      0 & \lambda^n & 2(a+b\sqrt{2}) \\
      0 & 0 & 1
   \end{pmatrix}
\end{equation*}
This group preserves a copy of $\H^2$, so it preserves a properly convex domain. However, it is still far from being discrete. In order to make it discrete, we just have to combine it with its Galois conjugate, by having them act on the two summands of the decomposition $\R^6=\R^3\oplus\R^3$. Observe that the Galois conjugate of $\lambda$ is $\lambda^{-1}$. Since there are going to be two $1$ on the diagonal, on the last position, we identify them by restricting to the subspace $\gen{e_1,e_2,e_4,e_5,e_3+e_6}$, giving
\begin{equation*}
   \begin{pmatrix}
      \lambda^{2n} & \lambda^n(a+b\sqrt{2}) & 0 & 0 & (a+b\sqrt{2})^2 \\
      0 & \lambda^n & 0 & 0 & 2(a+b\sqrt{2}) \\
      0 & 0 & \lambda^{-2n} & \lambda^{-n}(a-b\sqrt{2}) & (a-b\sqrt{2})^2 \\
      0 & 0 & 0 & \lambda^{-n} & 2(a-b\sqrt{2}) \\
      0 & 0 & 0 & 0 & 1
   \end{pmatrix}
\end{equation*}
Now, observe that there are two eigenvectors $\lambda^{2n}$ and $\lambda^{-2n}$, and that the two corresponding lines contain dominating entries. We rearrange the lines and the columns so that we have a diagonal block in the top left corner.
\begin{equation*}
   \begin{pmatrix}
      \lambda^{2n} & 0 & \lambda^n(a+b\sqrt{2}) & 0 & (a+b\sqrt{2})^2 \\
      0 & \lambda^{-2n} & 0 & \lambda^{-n}(a-b\sqrt{2}) & (a-b\sqrt{2})^2 \\
      0 & 0 & \lambda^n & 0 & 2(a+b\sqrt{2}) \\
      0 & 0 & 0 & \lambda^{-n} & 2(a-b\sqrt{2}) \\
      0 & 0 & 0 & 0 & 1
   \end{pmatrix}
\end{equation*}
We are not far from what we need now. Observe that this group does not preserve algebraic horospheres since it does not have any pair of fixed hyperplane and point with the same eigenvalue. For this reason, we add a column with a $1$ on the diagonal in the diagonal part.
\begin{equation*}
   \begin{pmatrix}
      \lambda^{2n} & 0 & 0 & \lambda^n(a+b\sqrt{2}) & 0 & (a+b\sqrt{2})^2 \\
      0 & \lambda^{-n} & 0 & 0 & \lambda^{-n}(a-b\sqrt{2}) & (a-b\sqrt{2})^2 \\
      0 & 0 & 1 & 0 & 0 & 0 \\
      0 & 0 & 0 & \lambda^n & 0 & 2(a+b\sqrt{2}) \\
      0 & 0 & 0 & 0 & \lambda^{-n} & 2(a-b\sqrt{2}) \\
      0 & 0 & 0 & 0 & 0 & 1
   \end{pmatrix}
\end{equation*}
It seems possible that this representation acts as a generalized cusp group on a generalized cusp domain. However, in order to ease the proof, we instead glue a new parabolic block conjugated into $\PO(2,1)$, giving the following representation $\rho$ of $R\times\R$.
\begin{equation*}
   \begin{pmatrix}
      \lambda^{2n} & 0 & 0 & 0 & \lambda^n(a+b\sqrt{2}) & 0 & (a+b\sqrt{2})^2 \\
      0 & \lambda^{-2n}& 0 & 0 & 0 &\lambda^{-n}(a-b\sqrt{2})&(a-b\sqrt{2})^2 \\
      0 & 0 & 1 & m & 0 & 0 & m^2\\
      0 & 0 & 0 & 1 & 0 & 0 & 2m \\
      0 & 0 & 0 & 0 & \lambda^n & 0 & 2(a+b\sqrt{2}) \\
      0 & 0 & 0 & 0 & 0&\lambda^{-n}& 2(a-b\sqrt{2}) \\
      0 & 0 & 0 & 0 & 0 & 0 & 1
   \end{pmatrix}
\end{equation*}

Observe that the matrix coefficient given by $\alpha=\begin{psmallmatrix}1 & 1 & 1 & 0 & 0 & 0 & 0\end{psmallmatrix}$ and $x=(1,1,1,\allowbreak 0,\allowbreak 0,\allowbreak 0,1)$ is $P:a,b,n,m\mapsto\lambda^{2n}+\lambda^{-2n} + 1 + (a+b\sqrt{2})^2 + (a-b\sqrt{2})^2 + m^2$ which is positive and s-proper. Since this dominates all entries on the last four lines, and $\abs{\lambda^n(a+b\sqrt{2})}$ and $\abs{\lambda^{-n}(a-b\sqrt{2})}$ are both smaller than $P(a,b,n,m)/2$, $P$ is also generic, and Proposition \ref{prop:gpPreservesDomainNSC} shows that $\rho$ preserves a properly convex domain. The proof of the following proposition is adapted from that of Proposition \ref{prop:suffCondStrictlyConvCase}.

\begin{proposition}\label{prop:exOfGenCuspWithSolvableHolo}
   There is an analytic generalized cusp domain $\Omega\subset\P(\R^7)$ whose parabolic face is a $2$-dimensional simplex on which $\rho(R\times\R)$ acts as a generalized cusp group.
\end{proposition}
\begin{proof}

   We denote by $\Gamma$ the image of $\rho(R\times\R)$. Let $p=e_3$ and $\phi=e_7^\ast$. Observe that $\phi(p)=0$ and that both are eigenvectors for the action of $\Gamma$ on $\R^7$ and $(\R^7)^\ast$ respectively, and that the associated eigenvalue is $1$.

   We have already showed that $\Gamma$ preserves a properly convex domain. By Proposition \ref{prop:gpPreservesDomainNSC}, we can thus find $U\subset\R^7$ and $V\subset(\R^7)^\ast$ such that for any $(x,\alpha)\in U\times V$, $\gamma\in\Gamma\mapsto\alpha(\gamma\cdot x)$ is a negative generic s-proper linear combinations of the entries of $\Gamma$. We also choose an open subset $U_1$ of $U$ such that $\overline{U_1}\subset U$. Up to shrinking $U$, we may also suppose that $\overline{\P(U)}$ does not meet $H:=\P(\ker\phi)$, nor any of the coordinate hyperplanes. We can now define
   \begin{align*}
      \Omega_1 &= \Int\Conv(\Gamma\cdot\P(U_1)) \\
      \Omega_0 &= \Int\Conv(\Gamma\cdot\P(U))
   \end{align*}
   Notice that $\Omega_0\cap H=\varnothing$.

   Observe that the dominating entries of $\rho$ are exactly the entries on its first two lines, as well as the rightmost entry of its third line.
   
   Let $(x,\alpha)\in U\times V$. Note that $\gamma\in\Gamma\mapsto\alpha(\gamma\cdot x)$ is a generic s-proper linear combination of the entries of $\Gamma$, but also of the entries of $\gamma\mapsto \gamma\cdot x$. Therefore, the entries of $\gamma\mapsto \gamma\cdot x$ have the same order as the entries of $\Gamma$ by Proposition \ref{prop:genPLCOrder}. But for $i\geq 4$, $\gamma\mapsto(\gamma\cdot x)_i$ is a linear combination of entries of $\Gamma$ that are dominated, so it is dominated. Hence all dominating entries of $\gamma\mapsto\gamma\cdot x$ are among its first $3$ entries. Moreover, each of these $3$ entries is dominating along a subsequence (the first and second entries are dominating along the sequences given by $(a,b,n,m)=(0,0,n,0)$ with $n$ going to either $+\infty$ or $-\infty$, and the third entry is dominating along $(a,b,n,m)=(0,0,0,m)$ where $m\to+\infty$), so these are exactly the dominating entries of $\gamma\mapsto\gamma\cdot x$. Therefore, all accumulation points of $\gamma\cdot x$ as $\gamma\to\infty$ in $\Gamma$ lie in $\Delta :=\P(\R_{>0}e_1 + \R_{>0}e_2 +\R_{>0}e_3) \subset H$. We also showed that for any vertex $v$ of $\Delta$, there is a sequence $\gamma_k\to\infty$ in $\Gamma$ such that $\gamma_k x\to v$. As this holds for any $x\in U$, we may pick $x\in U_1$, so that $\Delta\subset\overline{\Omega_1}$. Since $\Omega_0\cap H=\varnothing$, we see that $\Delta\subset\partial\Omega_0\cap\partial\Omega_1$. It is clear that $\Delta$ is preserved by the action of $\Gamma$.

   Since $\overline{\P(U_1)}$ is a compact subset of $\Omega_0$, its orbit under $\Gamma$ meets $\partial\Omega_0$ only at $\Delta$ by Proposition \ref{prop:autGpActsProp} and the last paragraph. Therefore $\overline{\Omega_1}-\Delta\subset\Omega_0$, so $\overline{\Omega_1}\cap H = \Delta$.

   Since $\rho$ takes values in $\SL(7,\R)$, the weight of $\Gamma$ associated to $H$ is trivial, so we can apply Lemma \ref{lm:smoothing} to $\Omega_1$ and $H$, yielding an analytic generalized cusp domain $\Omega\subset\Omega_1$ with parabolic face $\Delta$ which is invariant under the action of $\Gamma$.

   Observe that if $\xi := [e_3]=[p]$, then $\xi\in H$ is a vertex of $\Delta$, and both the weights of $\Gamma$ associated to $\xi$ and $H$ are trivial. Therefore, $\Gamma$ also preserves the algebraic horospheres centered at $(\xi,H)$ by Proposition \ref{prop:autPreservesAlgHoroNSC}. Hence $\Gamma$ acts on $\Omega$ as a generalized cusp group.
\end{proof}

\printbibliography
\end{document}